\documentclass[11pt]{elsarticle}

\makeatletter
\def\ps@pprintTitle{%
 \let\@oddhead\@empty
 \let\@evenhead\@empty
 \def\@oddfoot{}%
 \let\@evenfoot\@oddfoot}

\usepackage[left=2.5cm,right=2.5cm,top=2cm,bottom=2cm]{geometry}
\usepackage{amsfonts,amssymb,amsmath,amsthm,bm}
\usepackage{enumerate}
\usepackage{xcolor}
\usepackage{graphicx,float,epstopdf}
\usepackage{hyperref}\hypersetup{pdfauthor=author}
\usepackage{multirow}

\numberwithin{equation}{section}

\newtheorem{definition}{Definition}[section]
\newtheorem{theorem}[definition]{Theorem}
\newtheorem{lemma}[definition]{Lemma}
\newtheorem{corollary}[definition]{Corollary}
\newtheorem{example}[definition]{Example}
\newtheorem{remark}[definition]{Remark}

\DeclareMathOperator\spn{span}

\begin{document}

\begin{frontmatter}

\title{\LARGE A constrained gentlest ascent dynamics and its applications to finding excited states of Bose--Einstein condensates}

\author[hunnu,scnu]{Wei Liu}
\ead{wliu@m.scnu.edu.cn}

\author[hunnu]{Ziqing Xie}
\ead{ziqingxie@hunnu.edu.cn}

\author[hunnu]{Yongjun Yuan\corref{cor}}
\ead{yuanyongjun0301@163.com}
\cortext[cor]{Corresponding author.}

\address[hunnu]{Key Laboratory of Computing and Stochastic Mathematics (Ministry of Education), School of Mathematics and Statistics, Hunan Normal University, Changsha, Hunan 410081, PR China}

\address[scnu]{\emph{Present address}: South China Research Center for Applied Mathematics and Interdisciplinary Studies, South China Normal University, Guangzhou 510631, PR China}

\begin{abstract}
In this paper, the gentlest ascent dynamics (GAD) developed in W. E and X. Zhou (2011) \cite{EZ2011NL} is extended to a constrained gentlest ascent dynamics (CGAD) to find constrained saddle points with any specified Morse indices. It is proved that the linearly stable steady state of the proposed CGAD is exactly a nondegenerate constrained saddle point with a corresponding Morse index. Meanwhile, the locally exponential convergence of an idealized CGAD near nondegenerate constrained saddle points with corresponding indices is also verified. The CGAD is then applied to find excited states of single-component Bose--Einstein condensates (BECs) in the order of their Morse indices via computing constrained saddle points of the corresponding Gross--Pitaevskii energy functional under the normalization constraint. In addition, properties of the excited states of BECs in the linear/nonlinear cases are mathematically/numerically studied. Extensive numerical results are reported to show the effectiveness and robustness of our method and demonstrate some interesting physics. 
\end{abstract}

\begin{keyword}
constrained saddle points, 
constrained gentlest ascent dynamics, 
linear stability, 
Bose--Einstein condensates, 
excited states
\end{keyword}

\end{frontmatter}

\section{Introduction} 
Saddle points appear widely in various scientific fields as, for example, excited states in atomic, molecular and optical systems or transition states in chemical reactions. Particularly, the index-1 saddle point is a central concept in the study of rare events, which corresponds to the transition state between metastable states in randomly perturbed system \cite{ERV2002PRB,EZ2011NL}. In practice, excited states in some scenarios only occur instantaneously. And, transition states usually occur with very low probability. Owning to these difficulties in direct experimental observation, the effective numerical search of saddle points has attracted more and more attentions. Different numerical algorithms for finding saddle points have been carried out in the literature in recent decades, most of which are related to unconstrained saddle points. However, many physical/chemical/biological systems in practical scientific problems are constrained by one or more physical constraints, e.g., the wave function of a Bose--Einstein condensate (BEC) is constrained by one or more normalization conditions \cite{DGPS1999RMP,BC2013KRM}.  And, the volume and surface area of a biological vesicle membrane are fixed to be prescribed constants in the phase field model \cite{DLW2004JCP,CS2018SISC}. This motivates us to concern finding constrained saddle points.

In terms of numerical methods for finding unconstrained saddle points of given nonconvex energy functionals or multiple unstable solutions of nonlinear partial differential equations, we refer to the mountain-pass algorithm \cite{CM1993NATMA}, the high-linking algorithm \cite{DCC1999NA}, the local minimax method (LMM) \cite{LZ2001SISC}, the search extension method \cite{CX2004CMA}, the bifurcation method \cite{YLZ2008SCSA}, the string method \cite{ERV2002PRB}, the gentlest ascent dynamics (GAD) \cite{EZ2011NL}, the dimer method \cite{HJ1999JCP} and the shrinking dimer dynamics (SDD) \cite{ZD2012SINUM}, etc. Typically, the GAD developed by E and Zhou \cite{EZ2011NL} is a continuous dynamical system that describes the escape from the attractive basins of stable invariant sets. It is proved that the linearly stable steady state of the GAD proposed in \cite{EZ2011NL} is exactly an index-1 saddle point. And, due to its simplicity and effectiveness, the GAD has been applied to compute index-1 saddle points in many problems \cite{LZZ2013MMS,LLY2015JCP,ZRSD2016npj}. Several variants of the GAD such as the iterative minimization algorithm \cite{GLZ2016JCP} and the multiscale GAD \cite{GZ2017CMS}, were presented in literature. In \cite{QB2014TCA}, Quapp and Bofill proposed a generalized GAD algorithm that can compute unconstrained high-index saddle points. In addition, the SDD proposed by Zhang and Du \cite{ZD2012JCP} is closely related to the GAD. In fact, the SDD can be obtained by approximating the Hessian in the formulation of the GAD with first-order derivatives and introducing an additional dynamics for shrinking the length of the so-called dimer. Recently, Yin, Zhang and Zhang \cite{YZZ2019SISC} extended the SDD to find unconstrained high-index saddle points, and proposed a high-index optimization-based shrinking dimer (HiOSD) method.

There have existed several effective numerical methods in the literature to find constrained saddle points. In \cite{DZ2009CMS}, Zhang and Du proposed a constrained string method to compute the minimum energy path (MEP) with given constraints. In this way, the index-1 constrained saddle point given by the local maximizer of the energy functional on the MEP can be obtained accordingly. In a subsequent work of \cite{ZD2012SINUM}, Zhang and Du also proposed a constrained SDD (CSDD) \cite{ZD2012JCP} to search index-1 constrained saddle points. In \cite{LLY2015JCP}, Li, Lu and Yang modified the GAD to find index-1 saddle points of the Kohn--Sham density functional under the orthonormality constraints. Other numerical methods for finding constrained saddle points include the LMM based on the Rayleigh quotient or the active Lagrangian \cite{YZ2007SISC,YZ2008SISC}, the LMM using virtual geometric objects \cite{VGOLMM2}, and the Ljusternik--Schnirelman minimax algorithm \cite{Yao2019ACM}. These methods can be regarded as the variants of the original LMM developed by Li and Zhou in \cite{LZ2001SISC} and corresponding two-level optimization problems have to be solved. In summary, the above mentioned methods are mainly used to compute index-1 constrained saddle points or their efficiency are needed to be further improved. Thus, efficient numerical methods as well as the corresponding theoretical analysis are still called for to compute the general high-index constrained saddle points.

One of the important applications of computing constrained saddle points is to find the excited states of BECs. The BEC was first realized experimentally in dilute weakly interacting gases in 1995 \cite{AEMWC1995Science,BSTH1995PRL,DMADDKK1995PRL}. As is known, one of the basic problems in numerical studies of BEC is to determine the stationary states, i.e., the critical points of the energy functional under certain normalization constraints, by the mean field Gross--Pitaevskii (GP) theory. In the physics literatures, the stationary state with the lowest energy is called the ground state of BEC, whereas the stationary states with higher energies are usually called excited states. In the past two decades, based on the Gross--Pitaevskii equations (GPEs), many effective numerical methods for computing the ground states of BECs have been developed, as reviewed in, e.g., \cite{BC2013KRM}. However, the numerical methods for finding excited states of BECs are still relatively limited. The normalized gradient flow or the imaginary time evolution method \cite{BD2004SISC}, as one of the most popular techniques for computing the ground states of BECs, has been extended to compute the `first' excited states of single-component BECs with symmetries, see, e.g., \cite{BD2004SISC,ABDR2017CPC}. In addition, some continuation algorithms \cite{CC2007CPC,CCW2011JCP} and Newton-based iterative algorithm \cite{MGL2013CPC} are also designed to compute excited states of BECs. However, the convergence of these methods depend on the choice of initial data, and more efficient and accurate methods to compute excited states of BECs are still worthwhile explored.

In this paper, we are interested in developing a continuous dynamical system to stably search for constrained saddle points with any specified Morse indices. Due to the difficulties caused by constraints, instability, nonlinearity and nonconvexity, it is quite challenging to find constrained saddle points with general constraints in a stable way, especially for high-index ones. Inspired by the works of the original GAD \cite{EZ2011NL} for index-1 unconstrained saddle points and the CSDD \cite{ZD2012JCP} for index-1 constrained saddle points, we are aimed to propose a constrained gentlest ascent dynamics (CGAD) to compute general constrained saddle points with any specified indices and analyze its linear stability and local convergence. Further, we apply the CGAD to simulate excited states of BECs to demonstrate its effectiveness and robustness and then illustrate an interesting problem, i.e., the relation among the GP energies, chemical potentials and Morse indices of the excited states (as constrained saddle points) of BECs. In fact, it was found numerically that both the GP energy and chemical potential of the excited state increase with the increase of its Morse index, whereas the excited states with the same index may be at different energy levels.

The paper is organized as follows. In section~\ref{sec:cspmi}, we describe the definitions of constrained saddle points and their Morse indices. In section~\ref{sec:cgad}, we briefly review the original GAD and construct the CGAD to search for index-$k$ constrained saddle points. In section~\ref{sec:stab}, the mathematical justifications of the CGAD, including the linear stability and the local convergence of an idealized CGAD, are analyzed. In section~\ref{sec:esbec}, the CGAD is implemented to find some excited states of single-component BECs. Several interesting mathematical properties of excited states and the detailed numerical results in 1D and 2D are presented. Finally, some conclusions are drawn in section~\ref{sec:conclusion}.

\section{Constrained saddle points and Morse indices}\label{sec:cspmi}

Let $X$ be a real Hilbert space with its inner product $\langle\cdot,\cdot\rangle$ and norm $\|\cdot\|$. An energy functional $E\in C^2(X,\mathbb{R})$ and $m$ constraint functionals $G_i\in C^2(X,\mathbb{R})$, $i=1,2,\ldots,m$ are given. Consider critical points of the energy functional $E$ under constraints
\begin{align}\label{eq:m-constraints}
G_i(u)=0,\quad i=1,2,\ldots,m.
\end{align}
Denote $\mathcal{M}=\{u\in X: G_i(u)=0,i=1,2,\ldots,m\}$ as the constraint manifold. 

\begin{definition}
$u^*\in X$ is called a constrained critical point of $E$ on the manifold $\mathcal{M}$, or a constrained critical point of $E$ under the constraints \eqref{eq:m-constraints}, if there exist $\mu_i^*\in\mathbb{R}$, $i=1,2,\ldots,m$, such that
\begin{equation}\label{eq:csaddle-def}
  E'(u^*)-\sum_{i=1}^m\mu_i^* G_i'(u^*)=0, \qquad G_i(u^*)=0, \quad i=1,2,\ldots,m,
\end{equation}
where $E'$ and $G_i'$ represent the Fr\'{e}chet derivatives (or gradients) of $E$ and $G_i$, respectively. The constrained critical point that is not local extremizer (i.e., maximizer or minimizer) is called a constrained saddle point. 
\end{definition}

Throughout this paper, we assume that the constraints \eqref{eq:m-constraints} are {\em regular}, i.e., their gradients $G_i'(u)$, $i=1,2,\ldots,m$, are linearly independent for all $u\in\mathcal{M}$. Then $\mathcal{M}$ is a $C^2$ differential manifold, and its tangent space at $u\in\mathcal{M}$ is given by $T_u\mathcal{M}=\{v\in X:\langle G_i'(u),v\rangle=0,i=1,2,\ldots,m\}$. A direct computation shows that the orthogonal projection operator from $X$ onto the tangent space $T_u\mathcal{M}$ at $u\in\mathcal{M}$ takes
\begin{align}\label{eq:Pu-def}
P_u=I-\sum_{i=1}^m\sum_{j=1}^m g_{ij}(u)\left[G_i'(u)\otimes G_j'(u)\right],
\end{align}
where $I$ is the identity operator, $g_{ij}(u)$ are the $(i,j)$-elements of the inverse to the (positive definite) Gram matrix $[\langle G_i'(u),G_j'(u)\rangle]_{i,j=1,2,\ldots,m}$, and $\otimes$ denotes the tensor product operator defined as $(v\otimes w)\xi = \langle w,\xi\rangle v$, $\forall v,w,\xi\in X$. The projected gradient of $E$ at $u\in\mathcal{M}$ can be written as
\begin{align}\label{eq:projgrad-def}
F(u):=P_uE'(u)=E'(u)-\sum_{i=1}^m \mu_i(u) G_i'(u),
\end{align}
with $\mu_i(u)=\sum_{j=1}^m g_{ij}(u)\big\langle G_j'(u),E'(u)\big\rangle$, $i=1,2,\ldots,m$. Clearly, $u^*\in X$ is a constrained critical point of $E$ on $\mathcal{M}$ if and only if $G_i(u^*)=0$, $i=1,2,\ldots,m$, and $F(u^*)=0$.

For $u\in\mathcal{M}$, denoting $H(u):=E''(u)-\sum_{i=1}^m \mu_i(u) G_i''(u)$ the effective Hessian operator \cite{Luenberger1969}, we define the projected Hessian operator
\begin{align}\label{eq:Hhat-def}
\hat{H}(u)=P_uH(u)P_u:\quad T_u\mathcal{M}\to T_u\mathcal{M},
\end{align}
which is a self-adjoint linear operator on the tangent space $T_u\mathcal{M}$. Similar to the concept of Morse indices for unconstrained critical points \cite{Chang1993}, the stability/instability of a constrained critical point $u\in\mathcal{M}$ can be depicted by examining the spectrum of the linear operator $\hat{H}(u)$. More precisely, we introduce the following definition.

\begin{definition}
Assume that $ u^*$ is a constrained critical point of $E$ on the manifold $\mathcal{M}$. Let $T_{u^*}\mathcal{M}=T^-\oplus T^0\oplus T^+$, $\dim(T^0)<\infty$, where $T^-$, $T^0$ and $T^+$ are, respectively, the maximum negative, null, and maximum positive subspaces according to the spectral decomposition of the linear operator $\hat{H}(u^*):T_{u^*}\mathcal{M}\to T_{u^*}\mathcal{M}$. The {\em Morse index} of $u^*$ is defined as $\mathrm{index}(u^*)=\dim(T^-)$. $u^*$ is nondegenerate if $T^0=\{0\}$. Otherwise, $u^*$ is degenerate and $\dim(T^0)$ is called its nullity. When $\mathrm{index}(u^*)=k$ ($k=1,2,\ldots$), $u^*$ is called an index-$k$ constrained saddle point, $T^-$ is called its unstable (tangent) subspace and each nonzero vector in $T^-$ is called an unstable (tangent) direction at $u^*$.
\end{definition}

\section{The constrained gentlest ascent dynamics}\label{sec:cgad}


\subsection{Review of the GAD}\label{sec:review}

To propose our CGAD method, we first review the GAD developed in \cite{EZ2011NL} for finding index-1 unconstrained saddle points of $E$, which is formulated as
\begin{equation}\label{eq:EZ-GAD}
\left\{
  \begin{aligned}
    \dot{u} &= -E'(u)+2\langle E'(u),v\rangle v,\\
    \dot{v} &= -E''(u)v+\langle E''(u)v,v\rangle v,
  \end{aligned}\right.
\end{equation}
starting at $(u(0),v(0))=(u_0,v_0)\in X^2$ with $v_0$ satisfying the normalization condition $\|v_0\|=1$. Compared to the steepest descent dynamics or gradient flow
\begin{equation}\label{eq:GF}
\dot{u}=-E'(u),
\end{equation}
which works for finding local minima, the GAD \eqref{eq:EZ-GAD} consists of two equations. The first equation in \eqref{eq:EZ-GAD} can be obtained by performing the Householder transformation for the gradient flow with respect to the auxiliary unit vector $v$, where the last term in it makes $v$ a stable direction. 
The second equation in \eqref{eq:EZ-GAD}, evolving the vector $v$, is constructed by solving the Rayleigh quotient minimization problem $\min_{\|v\|=1}\langle E''(u)v,v\rangle$, which makes $v$ approximate the unstable direction of the target index-1 saddle point. %
It was proved in \cite{EZ2011NL} that, for an appropriately smooth energy function $E$ defined on an Euclidean space, the linearly stable steady state of the GAD \eqref{eq:EZ-GAD} is exactly an index-1 saddle point of $E$.

\subsection{CGAD for index-1 constrained saddle points}

The aim of this section is to propose the CGAD for finding constrained saddle points with any specified indices. To clarify the idea, we first construct the formulation of the CGAD to search for index-1 constrained saddle points.

Let $u\in\mathcal{M}$ be an approximation of an index-1 constrained saddle point of $E$ on the constraint manifold $\mathcal{M}=\{u\in X: G_i(u)=0,i=1,2,\ldots,m\}$, and the unit vector $v\in T_u\mathcal{M}$ be an approximation of the corresponding unstable tangent direction, see Fig.~\ref{fig:cgad1-idea}~(left). We discuss below how to construct the evolution equations of $u$ and $v$.

\begin{figure}[!ht]
\centering
\includegraphics[width=.33\textwidth]{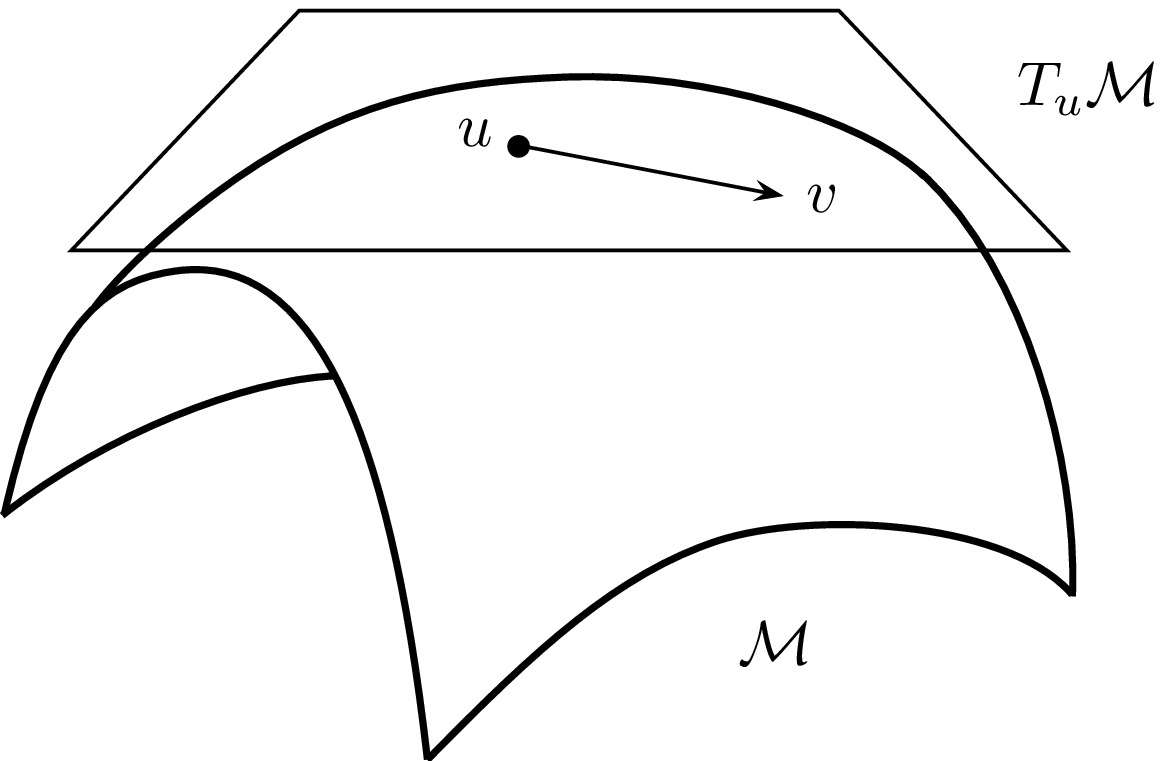} \quad
\includegraphics[width=.35\textwidth]{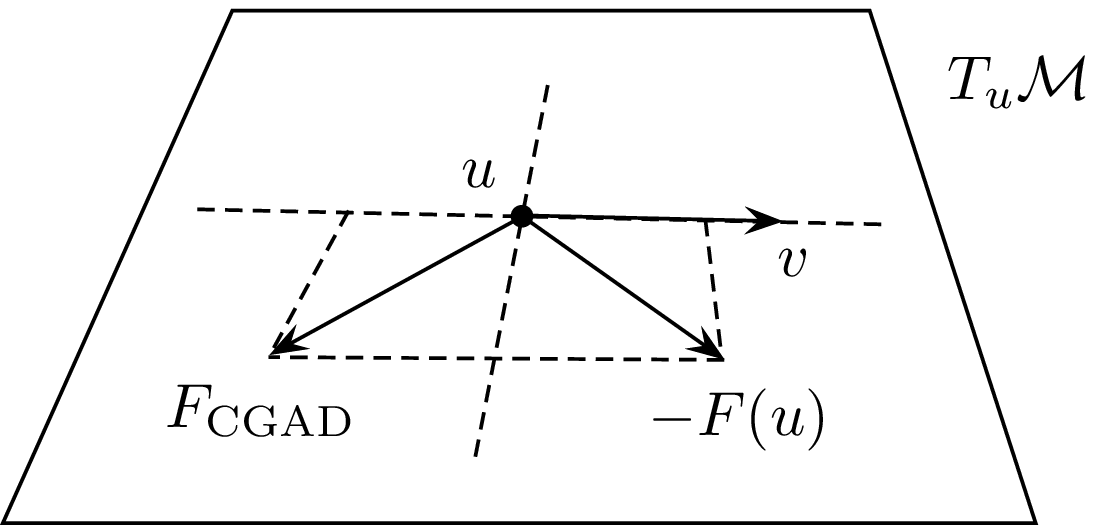}
\caption{Illustration of the index-1 CGAD. Left: $u\in\mathcal{M}$ and $v\in T_u\mathcal{M}$ approximate an index-1 constrained saddle point on the manifold $\mathcal{M}$ and its unit unstable direction, respectively. Right: the projected gradient $F(u)=P_uE'(u)$ has the decomposition $F(u)=F_{v}(u)+F_{\bot}(u)$ with $F_v(u)\in T_u\mathcal{M}\cap\mathrm{span}\{v\}$ and $F_{\bot}(u)\in T_u\mathcal{M}\cap\mathrm{span}\{v\}^\bot$, and thus, the force of the index-1 CGAD to evolve $u$ is constructed as $F_{\mathrm{CGAD}}=F_{v}(u)-F_{\bot}(u)=-F(u)+2F_{v}(u)$.}
\label{fig:cgad1-idea}
\end{figure}

\begin{itemize}
\item {\bfseries Construction of the dynamics for $u$.} 
To guarantee that $u$ moves towards an index-1 constrained saddle point, the evolution of $u$ in the direction $v$ has to increase the energy, while the evolution in other directions decreases the energy. Moreover, to preserve the constraint $u\in\mathcal{M}$ (i.e., $G_i(u)=0$, $i=1,2,\ldots,m$), the force to evolve $u$ must be in the tangent space $T_u\mathcal{M}$. Thus, we construct the dynamics for $u$ as
\begin{equation}\label{eq:cgad-idx1-eq1idea}
  \dot{u}=F_{v}(u)-F_{\bot}(u),
\end{equation}
where $F_{v}(u)=\langle F(u),v\rangle v$ is the component of the projected gradient $F(u)=P_uE'(u)$ in $v$ and $F_{\bot}(u)=F(u)-\langle F(u),v\rangle v$ the component of $F(u)$ in the orthogonal complement of $v$, as illustrated in Fig.~\ref{fig:cgad1-idea}~(right). Intuitively, the first term in \eqref{eq:cgad-idx1-eq1idea} makes the energy increase in $v$ and the second term makes the energy decrease in other directions. 
\item {\bfseries Construction of the dynamics for $v$.} 
From the definition of unstable directions, if $u$ is an index-1 constrained saddle point, its unstable direction $v$ is an eigenvector of the projected Hessian $\hat{H}(u)$ corresponding to the unique negative eigenvalue. By the Rayleigh-Ritz variational principle \cite{Zeidler1985III}, $v$ can be obtained by solving the following minimization problem
\begin{equation}\label{eq:cgad-idx1-eq2idea-min}
   \min_{v\in T_u\mathcal{M},\|v\|^2=1}\langle \hat{H}(u)v, v\rangle.
\end{equation}
Considering the Lagrangian 
\begin{align}
\mathcal{L}(v,\lambda,\bar{\lambda}_1,\bar{\lambda}_2,\ldots,\bar{\lambda}_m)
= \frac12\langle \hat{H}(u)v, v\rangle-\frac{\lambda}{2}(\|v\|^2-1)-\sum_{i=1}^m\bar{\lambda}_i\langle G_i'(u),v\rangle, 
\end{align}
we construct the dynamics for $v$ as
\begin{equation}\label{eq:cgad-idx1-eq2idea}
  \dot{v}=-\frac{\delta}{\delta v}\mathcal{L}(v,\lambda,\bar{\lambda}_1,\bar{\lambda}_2,\ldots,\bar{\lambda}_m)
   = -\hat{H}(u)v+\lambda v+\sum_{i=1}^m\bar{\lambda}_i G_i'(u),
\end{equation}
where $\lambda=\lambda(u,v)$ and $\bar{\lambda}_i=\bar{\lambda}_i(u,v)$ are the Lagrange multipliers corresponding to the constraints $\|v\|^2=1$ and $\langle G_i'(u),v\rangle=0$, $i=1,2,\ldots,m$ (i.e., $v\in T_u\mathcal{M}$), respectively.
\end{itemize}

In summary, the CGAD for finding an index-1 constrained saddle point is formulated as
\begin{equation}\label{eq:cgad-idx1}
\left\{
\begin{aligned}
  \dot{u} &= - F(u) + 2\langle F(u),v\rangle v, \\
  \dot{v} &= - \hat{H}(u)v + \lambda v +\sum_{i=1}^m\bar{\lambda}_iG_i'(u),
\end{aligned}\right.
\end{equation}
with the initial data $(u(0),v(0))=(u_0,v_0)$ satisfying $u_0\in\mathcal{M}$, $v_0\in T_{u_0}\mathcal{M}$ and $\|v_0\|^2=1$. The Lagrange multipliers $\lambda$ and $\bar{\lambda}_i$ ($i=1,2,\ldots,m$) in \eqref{eq:cgad-idx1} are chosen such that the flow preserves the constraints $\|v\|^2=1$ and $\langle G_i'(u),v\rangle=0$ ($i=1,2,\ldots,m$), respectively. Therefore, $\langle v,\dot{v}\rangle=0$ and $\langle G_i'(u),\dot{v}\rangle+\langle G_i''(u)\dot{u},v\rangle=0$, which lead to
\begin{equation*}
  \lambda = \langle\hat{H}(u)v,v\rangle,\quad
  \bar{\lambda}_i = \sum_{j=1}^m g_{ij}(u)\Big\langle G_j''(u)v,\, F(u) - 2\langle F(u),v\rangle v\Big\rangle,\quad
  i=1,2,\ldots,m.
\end{equation*}

\subsection{CGAD for high-index constrained saddle points}

Now, we extend the index-1 CGAD \eqref{eq:cgad-idx1} to general high-index cases. 
To construct the CGAD for finding an index-$k$ ($k=1,2,3,\ldots$) constrained saddle point of the energy functional $E\in C^2(X,\mathbb{R})$ on the constraint manifold $\mathcal{M}$, we need to consider $k$ linearly independent unstable tangent directions $v_1,v_2,\ldots,v_k\in T_u\mathcal{M}$ (see Fig.~\ref{fig:cgadk-idea}). Let $u\in\mathcal{M}$ be an approximation of an index-$k$ constrained saddle point and $V=\spn\{v_1,v_2,\ldots,v_k\}$ the approximation of corresponding unstable subspace. Denote $W$ by the orthogonal complement of $V$ in $T_u\mathcal{M}$. 
\begin{figure}[!ht]
\centering
\includegraphics[width=.4\textwidth]{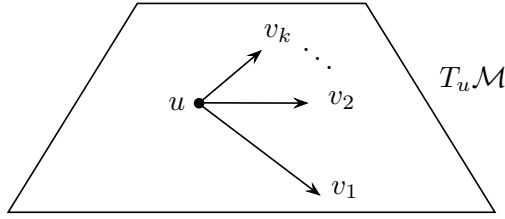}
\caption{Illustration of unstable directions $v_1,v_2,\ldots,v_k\in T_u\mathcal{M}$ for an index-$k$ constrained saddle point $u\in\mathcal{M}$.}
\label{fig:cgadk-idea}
\end{figure}

In order to make $u$ move towards an index-$k$ constrained saddle point, the force that evolves $u$ needs to be in the tangent space $T_u\mathcal{M}$ with its components in $V$ and $W$ increasing and decreasing the energy, respectively. It is natural to evolve $u$ by the steepest ascent dynamics in $V$ and the steepest descent dynamics in $W$, i.e., the dynamics for $u$ is as
\begin{equation}\label{eq:hicgad-eq1idea}
  \dot{u}=F_V(u)-F_W(u),
\end{equation}
where $F_V(u)$ and $F_W(u)=F(u)-F_V(u)$ are the orthogonal projections of the projected gradient $F(u)=P_uE'(u)$ on $V$ and $W$, respectively. If $v_1,v_2,\ldots,v_k\in T_u\mathcal{M}$ satisfy the orthonormal conditions: $\langle v_i,v_j\rangle=\delta_{ij}$, then $F_V(u)=\sum_{j=1}^k\langle F(u),v_j\rangle v_j$, and \eqref{eq:hicgad-eq1idea} becomes
\begin{equation}\label{eq:hicgad-eq1idea1}
  \dot{u}=-F(u)+2F_V(u)
  =-F(u)+2\sum_{j=1}^k\langle F(u),v_j\rangle v_j.
\end{equation}

It is worthwhile to point out that, if $u$ is an index-$k$ constrained saddle point, its unstable directions $v_1,v_2,\ldots,v_k$ can be taken as the orthonormal eigenvectors of the projected Hessian $\hat{H}(u)$ corresponding to the $k$ smallest and negative eigenvalues. By the Rayleigh-Ritz variational principle \cite{Zeidler1985III}, the eigenvector $v_1$ corresponding to the smallest eigenvalue can be obtained by minimizing $\langle\hat{H}(u)v_1,v_1\rangle$ under the constraints $v_1\in T_u\mathcal{M}$ and $\|v_1\|^2=1$. And, when eigenvectors $v_1,v_2,\ldots,v_{i-1}$ of $\hat{H}(u)$ corresponding to the first $i-1$ smallest eigenvalues are known, the eigenvector $v_i$ corresponding to the $i$-th smallest eigenvalue can be obtained by solving the following Rayleigh-Ritz minimization problem \cite{Zeidler1985III}
\begin{equation}\label{eq:hicgad-eq2ideaminvi}
  \min_{v_i} \langle\hat{H}(u)v_i,v_i\rangle\quad \mbox{s.t.}\quad v_i\in T_u\mathcal{M},\quad \langle v_i,v_j\rangle=\delta_{ij},\quad j=1,2,\ldots,i.
\end{equation}
Consider the Lagrangian
\begin{align*}
  & \mathcal{L}_i(v_i,\lambda_{i1},\lambda_{i2},\ldots,\lambda_{ii},\bar{\lambda}_{i1},\bar{\lambda}_{i2},\ldots,\bar{\lambda}_{im}) \\
  &\qquad= \frac12 \langle\hat{H}(u)v_i,v_i\rangle-\frac{\lambda_{ii}}{2}\left(\|v_i\|^2-1\right) - \sum_{j=1}^{i-1}\lambda_{ij}\langle v_i,v_j\rangle -\sum_{l=1}^m\bar{\lambda}_{il} \langle G_l'(u),v_i\rangle,
\end{align*}
with Lagrange multipliers $\lambda_{ij}$ ($j=1,2,\ldots,i$) and $\bar{\lambda}_{il}$ ($l=1,2,\ldots,m$) corresponding to constraints $\langle v_i,v_j\rangle=\delta_{ij}$ and $\langle G_l'(u),v_i\rangle=0$ (i.e., $v_i\in T_u\mathcal{M}$), respectively. The gradient flow for solving \eqref{eq:hicgad-eq2ideaminvi} is given by the following dynamics for $v_i$ ($i=1,2,\ldots,k$):
\begin{equation*}
  \dot{v}_i = -\frac{\delta \mathcal{L}_i}{\delta v_i}
  = -\hat{H}(u) v_i +\sum_{j=1}^i \lambda_{ij}v_j + \sum_{l=1}^m\bar{\lambda}_{il} G_l'(u).
\end{equation*}

Based on the above discussion, we propose the following CGAD to search for the index-$k$ constrained saddle point:
\begin{align}\label{eq:cgad-idxk}
\left\{
\begin{aligned}
\gamma_0\dot{u} &= -F(u)+2\sum_{i=1}^k\langle F(u),v_i\rangle v_i, \\
\gamma_i\dot{v}_i &= -\hat{H}(u)v_i+\sum_{j=1}^i\lambda_{ij}v_j+\sum_{l=1}^m\bar{\lambda}_{il}G_l'(u),\quad i=1,2,\ldots,k.
\end{aligned}\right.
\end{align}
Here $\gamma_i>0$ ($i=0,1,\ldots,k$) are relaxation parameters. The Lagrange multipliers $\lambda_{ij}$ ($j=1,2,\ldots,i$) and $\bar{\lambda}_{il}$ ($l=1,2,\ldots,m$) are chosen such that the flow preserves the constraints $\langle v_i,v_j\rangle=\delta_{ij}$ and $\langle G_l'(u),v_i\rangle=0$, respectively, which leads to $\langle \dot{v}_i,v_j\rangle+\langle v_i,\dot{v}_j\rangle=0$ and $\langle G_l''(u)\dot{u},v_i\rangle+\langle G_l'(u),\dot{v}_i\rangle=0$. Hence
\begin{align}
\lambda_{ij} &=\left(1+\gamma_i/\gamma_j-\delta_{ij}\right)\langle \hat{H}(u) v_i,v_j\rangle,\quad j=1,2,\ldots,i, \label{eq:cgad-idxk-lambdaij} \\
\bar{\lambda}_{il} &= \frac{\gamma_i}{\gamma_0}\sum_{l'=1}^m g_{ll'}(u)\bigg\langle G_{l'}''(u)v_i,\, F(u) - 2\sum_{j=1}^k\langle F(u),v_j\rangle v_j\bigg\rangle,\quad l=1,2,\ldots,m, \label{eq:cgad-idxk-lambdaibar}
\end{align}
for $i=1,2,\ldots,k$. The initial data $(u(0),v_1(0),\ldots,v_k(0))$ of \eqref{eq:cgad-idxk} is assumed to satisfy $u(0)\in\mathcal{M}$, $v_i(0)\in T_{u(0)}\mathcal{M}$ and $\langle v_i(0),v_j(0)\rangle=\delta_{ij}$ for $1\leq j\leq i\leq k$, or equivalently,
\begin{subequations}\label{eq:cgad-idxk-init}
\begin{align}
  G_l(u(0)) &= 0, \quad l=1,2,\ldots,m, \label{eq:cgad-idxk-init1}\\
  \langle G_l'(u(0)),v_i(0)\rangle &= 0, \quad l=1,2,\ldots,m,\; i=1,2,\ldots,k, \label{eq:cgad-idxk-init2}\\
  \langle v_i(0),v_j(0)\rangle &= \delta_{ij},\; 1\leq j\leq i\leq k. \label{eq:cgad-idxk-init3}
\end{align}
\end{subequations}
Clearly, the index-1 CGAD \eqref{eq:cgad-idx1} is a special case of the CGAD \eqref{eq:cgad-idxk}.

\begin{remark}
According to the CGAD \eqref{eq:cgad-idxk}, the CSDD proposed in \cite{ZD2012JCP} for finding index-1 constrained saddle points can be easily extended to a high-index CSDD for searching for index-$k$ constrained saddle points. Actually, the approximation
\[ \hat{H}(u)v_i\approx\frac{F(u+\ell v_i)-F(u-\ell v_i)}{2\ell},\quad i=1,2,\ldots,k, \]
and an additional dynamics for shrinking the parameter $\ell>0$, e.g., $\dot{\ell}=-\ell$ \cite{ZD2012JCP,ZD2012SINUM}, should be implemented to construct the index-$k$ CSDD from the CGAD \eqref{eq:cgad-idxk}.
\end{remark}

The following lemma states that the CGAD \eqref{eq:cgad-idxk} preserves exactly the constraints \eqref{eq:cgad-idxk-init}. 
\begin{lemma}\label{lem:cgad-idxk-conpty}
Assume that $E,G_l\in C^2$, $l=1,2,\ldots,m$, and the constraints \eqref{eq:m-constraints} are regular. Let $(u(t),v_1(t),\ldots,v_k(t))$ be the solution of \eqref{eq:cgad-idxk} with the initial data satisfying \eqref{eq:cgad-idxk-init}. Then
\begin{subequations}\label{eq:cgad-idxk-conpty}
\begin{align}
  G_l(u(t)) &\equiv 0, \quad l=1,2,\ldots,m, \label{eq:cgad-idxk-conpty1}\\
  \langle G_l'(u(t)),v_i(t)\rangle &\equiv 0, \quad l=1,2,\ldots,m,\; i=1,2,\ldots,k, \label{eq:cgad-idxk-conpty2}\\
  \langle v_i(t),v_j(t)\rangle &\equiv \delta_{ij},\; 1\leq j\leq i\leq k. \label{eq:cgad-idxk-conpty3}
\end{align}
\end{subequations}
\end{lemma}
\begin{proof}
See detailed proof in \ref{sec:pf-cgad-idxk-conpty}.

\end{proof}

\section{Linear stability and local convergence}\label{sec:stab}

In this section, we study the stability and convergence of the CGAD \eqref{eq:cgad-idxk}. The following lemma will play important role in the subsequent analysis, with its proof detailed in \ref{sec:pf-dmudF}.

\begin{lemma}\label{lem:dmudF}
Assume that $E,G_i\in C^2$, $i=1,2,\ldots,m$, and the constraints \eqref{eq:m-constraints} are regular. Then, for $u\in \mathcal{M}$,
\begin{align}
F'(u)v &=\hat{H}(u)v-\sum_{i=1}^m\sum_{j=1}^m g_{ij}(u) \big\langle G_j''(u)F(u),v\big\rangle G_i'(u),\quad \forall v\in T_u\mathcal{M}.
\end{align}
In particular, $F'(u^*)=\hat{H}(u^*)$ if $u^*\in \mathcal{M}$ is a constrained critical point.
\end{lemma}

\subsection{Linear stability of the CGAD}

We now show that the linearly stable steady state of the CGAD \eqref{eq:cgad-idxk} is exactly a nondegenerate index-$k$ constrained saddle point of $E$ on the manifold $\mathcal{M}$. Similar results of the GAD for unconstrained saddle points and the CSDD for index-1 constrained saddle point can be found in \cite{EZ2011NL,YZZ2019SISC} and \cite{ZD2012JCP}, respectively.

\begin{theorem}\label{thm:cgad-idxk-linstab}
Assume that $E,G_l\in C^3$, $l=1,2,\ldots,m$, and the constraints \eqref{eq:m-constraints} are regular. Let $u^*\in\mathcal{M}$ and $v_i^*\in T_{u^*}\mathcal{M}$, $i=1,2,\ldots,k$, satisfy $\langle v_i^*,v_j^*\rangle=\delta_{ij}$.
\begin{enumerate}[(a)]
  \item $(u^*,v_1^*,\ldots,v_k^*)$ is a steady state of \eqref{eq:cgad-idxk} if and only if $u^*$ is a constrained critical point of $E$ on the manifold $\mathcal{M}$ and $v_i^*$, $i=1,2,\ldots,k$, are eigenvectors of $\hat{H}(u^*)$.
  \item $(u^*,v_1^*,\ldots,v_k^*)$ is a linearly stable steady state of \eqref{eq:cgad-idxk} if and only if the following hold:
  \begin{enumerate}[(i)]
    \item $u^*$ is a nondegenerate index-$k$ constrained saddle point of $E$ on the manifold $\mathcal{M}$; 
    \item all the eigenvalues of $\hat{H}(u^*)$, say $\{\lambda_i^*\}$, satisfy $\lambda_1^*<\lambda_2^*<\cdots<\lambda_k^*<0<\lambda_{k+1}^*\leq\lambda_{k+2}^*\leq\cdots$; 
    \item for $i=1,2,\ldots,k$, $v_i^*$ is the eigenvector of $\hat{H}(u^*)$ corresponding to the eigenvalue $\lambda_i^*$.
  \end{enumerate}
\end{enumerate}
\end{theorem}

\begin{proof}
(a) {\em Necessity.} Suppose that $(u^*,v_1^*,\ldots,v_k^*)$ is a steady state of \eqref{eq:cgad-idxk}, i.e.,
\begin{align}
-F(u^*) + 2\sum_{i=1}^k\langle F(u^*),v_i^*\rangle v_i^* &=0, \label{eq:thm-idxk-a:n1} \\
  -\hat{H}(u^*) v_i^* +\sum_{j=1}^i \lambda_{ij}^*v_j^* + \sum_{l=1}^m\bar{\lambda}_{il}^* G_l'(u^*) &= 0, \quad i=1,2,\cdots, k, 
  \label{eq:thm-idxk-a:n2}
\end{align}
where $\lambda_{ij}^*$ and $\bar{\lambda}_{il}^*$ are given in \eqref{eq:cgad-idxk-lambdaij}-\eqref{eq:cgad-idxk-lambdaibar} with $(u,v_1,\ldots,v_k)$ replaced by $(u^*,v_1^*,\ldots,v_k^*)$. Noting that $\langle v_i^*,v_j^*\rangle=\delta_{ij}$, taking the inner product in both sides of \eqref{eq:thm-idxk-a:n1} with $v_j^*$ yields $\langle F(u^*),v_j^*\rangle=0$, $j=1,2,\ldots,k$. Therefore, $F(u^*)=0$, i.e., $u^*$ is a constrained critical point. Moreover, by \eqref{eq:cgad-idxk-lambdaibar}, we have $\bar{\lambda}_{il}^*=0$, thus \eqref{eq:thm-idxk-a:n2} becomes
\begin{equation}\label{eq:thm-idxk-a:n2eq}
  \hat{H}(u^*) v_i^* = \sum_{j=1}^i \lambda_{ij}^*v_j^*, \quad i=1,2,\cdots, k. 
\end{equation}
Taking the inner product in both side of \eqref{eq:thm-idxk-a:n2eq} with $v_l^*$ ($\forall l<i$) implies that $\langle \hat{H}(u^*)v_i^*,v_l^*\rangle=\sum_{j=1}^i \lambda_{ij}^*\langle v_j^*,v_l^*\rangle=\lambda_{il}^*$, which states $\lambda_{il}^*=0$ (otherwise, one gets a contradiction with \eqref{eq:cgad-idxk-lambdaij}). Consequently, $\hat{H}(u^*)v_i^*=\lambda_{ii}^*v_i^*$, $i=1,2,\cdots,k$, i.e., $\{(\lambda_{ii}^*,v_i^*)\}_{i=1}^k$ are eigenpairs of $\hat{H}(u^*)$.

{\em Sufficiency.} Suppose that $u^*$ is a constrained critical point of $E$ on the manifold $\mathcal{M}$ and $v_i^*$ ($i=1,2,\ldots,k$) are eigenvectors of $\hat{H}(u^*)$. Then $F(u^*)=0$, \eqref{eq:thm-idxk-a:n1} is satisfied, and $\bar{\lambda}_{il}^*=0$. On the other hand, for each $i=1,2,\ldots,k$, since $v_i^*$ is an eigenvector of $\hat{H}(u^*)$, there exists an $\omega_i^*\in\mathbb{R}$ such that $\hat{H}(u^*) v_i^* = \omega_i^* v_i^*$. We have
\begin{align*}
\lambda_{ij}^*
=(1+\gamma_i/\gamma_j-\delta_{ij})\langle\hat{H}(u^*) v_i^*, v_j^*\rangle 
=(1+\gamma_i/\gamma_j-\delta_{ij})\,\omega_{i}^*\delta_{ij}
=\omega_{i}^*\delta_{ij}.
\end{align*}
 Thus, \eqref{eq:thm-idxk-a:n2} holds. Consequently, $(u^*,v_1^*,\ldots,v_k^*)$ is a steady state of \eqref{eq:cgad-idxk}.

(b) Consider the Jacobian operator of the right-hand-side of \eqref{eq:cgad-idxk}, denoted by $J$. Direct computations and the application of Lemma~\ref{lem:dmudF} show that
\begin{align*}
J=\begin{pmatrix}
     J_{00} & J_{01} & J_{02} & \cdots & J_{0k} \\
     J_{10} & J_{11} & 0 & \cdots & 0 \\
     J_{20} & J_{21} & J_{22} & \cdots & 0 \\
     \vdots & \vdots & \vdots & \ddots & \vdots \\
     J_{k0} & J_{k1} & J_{k2} & \cdots & J_{kk} \\
    \end{pmatrix},
\end{align*}
where
\begin{align*}
J_{00} = \frac{\delta \dot{u}}{\delta u}
 &= -\frac{1}{\gamma_0}\bigg(I-2\sum_{j=1}^kv_j\otimes v_j\bigg) \bigg(\hat{H}(u)-\sum_{l=1}^m\sum_{l'=1}^m g_{ll'}(u)\left[G_{l}'(u)\otimes F(u)\right]G_{l'}''(u)\bigg),\\
J_{0i} = \frac{\delta \dot{u}}{\delta v_i}
 &= \frac{2}{\gamma_0}\Big(\langle F(u),v_i\rangle I +v_i\otimes F(u)\Big), \\
J_{ii} = \frac{\delta \dot{v}_i}{\delta v_i}
 &= \frac{1}{\gamma_i}\bigg(-\hat{H}(u)+\lambda_{ii}I+\sum_{j=1}^i\left(1+\gamma_i/\gamma_j\right)(v_j\otimes v_j)\hat{H}(u)+\sum_{l=1}^m G_l'(u)\otimes\frac{\delta\bar{\lambda}_{il}}{\delta v_i}\bigg),
\end{align*}
for $i=1,2,\ldots,k$, with
\begin{align*}
\frac{\delta\bar{\lambda}_{il}}{\delta v_i} = \frac{\gamma_i}{\gamma_0}\sum_{l'=1}^m g_{ll'}(u) \bigg[ G_{l'}''(u)\Big(I-2\sum_{j=1}^kv_j\otimes v_j\Big)F(u)-2\Big(\!\langle F(u),v_i\rangle I +v_i\otimes F(u)\Big)G_{l'}''(u)v_i\bigg].
\end{align*}
Since at the steady state $(u^*,v_1^*,\ldots,v_k^*)$, $F(u^*)=0$ and $\hat{H}(u^*)v_i^*=\lambda_{i}^*v_i^*$ with $\lambda_{i}^*:=\lambda_{ii}^* =\langle \hat{H}(u^*) v_i^*,v_i^*\rangle$, $i=1,2,\ldots, k$, the Jacobian $J$ at $(u^*,v_1^*,\ldots,v_k^*)$ is a linear operator from $(T_{u^*}\mathcal{M})^{k+1}$ to $(T_{u^*}\mathcal{M})^{k+1}$ and takes a block lower triangular form with diagonal blocks
\begin{align*}
J_{00} &= \frac{1}{\gamma_0}\bigg(2\sum_{j=1}^k\lambda_{j}^*( v_j^*\otimes v_j^*)-\hat{H}(u^*)\bigg),\\
J_{ii} &= \frac{\lambda_{i}^*I-\hat{H}(u^*)}{\gamma_i} + \sum_{j=1}^i\bigg(\frac{1}{\gamma_i}+\frac{1}{\gamma_j}\bigg)\lambda_{j}^*( v_j^*\otimes v_j^*),\quad i=1,2,\ldots, k.
\end{align*}
Moreover, since $\{(\lambda_{i}^*,v_i^*)\}_{i=1}^k$ are eigenpairs of $\hat{H}(u^*)$, the diagonal blocks $J_{00}, J_{11}, \ldots, J_{kk}$ and $\hat{H}(u^*)$ share the same eigenvectors $v_i^* \ (i=1,2,\ldots, k)$. Let $\lambda_{k+1}^*\leq\lambda_{k+2}^*\leq\cdots$ be all other eigenvalues of $\hat{H}(u^*)$ and $v_{k+1}^*, v_{k+2}^*,\cdots\in T_{u^*}\mathcal{M}$ be the corresponding eigenvectors. Due to $\hat{H}(u^*)$ is self-adjoint and $\{v_j^*\}_{j=1}^k$ is orthonormal, one may assume that $\{v_j^*\}_{j\geq1}$ is an orthonormal system. It is calculated that
\begin{align*}
J_{00} v_l^*
  &= \frac{1}{\gamma_0}\bigg(2\sum_{j=1}^k\lambda_j^*\langle v_j^*,v_l^*\rangle v_j^*-\hat{H}(u^*)v_l^*\bigg)
  =\begin{cases}
     (\lambda_l^*/\gamma_0) v_l^*, & 1\leq l\leq k, \\
     -(\lambda_l^*/\gamma_0) v_l^*, & l> k,
   \end{cases}
  \\
J_{ii} v_l^*
  &= \frac{\lambda_i^*v_l^* -\hat{H}(u^*)v_l^*}{\gamma_i} + \sum_{j=1}^i\bigg(\frac{1}{\gamma_i}+\frac{1}{\gamma_j}\bigg)\lambda_j^*\langle v_j^*, v_l^*\rangle v_j^* 
  =\begin{cases}
     (\lambda_i^*/\gamma_i+\lambda_l^*/\gamma_l) v_l^*, &  1\leq l\leq i, \\
     ((\lambda_i^*-\lambda_l^*)/\gamma_i) v_l^*, &  l>i.
   \end{cases}
\end{align*}
Hence all eigenvalues of $J$ are given as
\[ \left\{\lambda_{i}^*/\gamma_0,\; -\lambda_l^*/\gamma_0,\; \lambda_{i}^*/\gamma_i+\lambda_{r}^*/\gamma_{r}, \; (\lambda_{i}^*-\lambda_{s}^*)/\gamma_i,\quad \forall\, 1\leq r\leq i\leq k<l,\; s>i \right\}. \] 
Thus, $(u^*,v_1^*,\ldots,v_k^*)$ is a linearly stable steady state of \eqref{eq:cgad-idxk} if and only if all eigenvalues of $J$ are negative, i.e., $\lambda_1^*<\lambda_2^*<\cdots<\lambda_k^*<0<\lambda_{k+1}^*\leq \lambda_{k+2}^*\leq\cdots$. Equivalently, (i) $u^*$ is a nondegenerate index-$k$ constrained saddle point; (ii) all eigenvalues of $\hat{H}(u^*)$ satisfy $\lambda_1^*<\lambda_2^*<\cdots<\lambda_k^*<0<\lambda_{k+1}^*\leq\lambda_{k+2}^*\leq\cdots$; and (iii) $v_i^*$ is the eigenvector of $\hat{H}(u^*)$ corresponding to $\lambda_i^*$, $i=1,2,\ldots,k$.
\end{proof}

\begin{remark}
For the index-1 case, the condition/conclusion (ii) in part (b) of Theorem~\ref{thm:cgad-idxk-linstab} does not need to appear in the theorem since it is implied by the nondegeneracy in (i).
\end{remark}

\subsection{Locally exponential convergence of an idealized CGAD}

Due to the complexity of constraints and nonlinearities, there are some potential difficulties in directly analyzing the global convergence of the CGAD \eqref{eq:cgad-idxk}. For simplicity, based on a similar idea to the study on an idealized version of the original GAD in \cite{LO2017SINUM}, we consider the following idealized CGAD:
\begin{equation}\label{eq:idealizedCGAD}
  \dot{u} = - F(u) + 2\sum_{i=1}^k\langle F(u),v_i(u)\rangle v_i(u), \quad u(0)=u_0\in\mathcal{M},
\end{equation}
where $v_i(u)\in T_u\mathcal{M}$, satisfying $\langle v_i(u),v_j(u)\rangle=\delta_{ij}$, are exact eigenvectors of the linear operator $\hat{H}(u):T_u\mathcal{M}\to T_u\mathcal{M}$ corresponding to the smallest $k$ eigenvalues $\lambda_1(u)\leq\lambda_2(u)\leq\cdots\leq\lambda_k(u)$.

Since $\dot{u}\in T_u\mathcal{M}$ by \eqref{eq:idealizedCGAD}, we have
\[ \frac{\mathrm{d}}{\mathrm{d} t}G_l(u)=\langle G_l'(u),\dot{u}\rangle=0, \quad l=1,2,\ldots,m. \] 
The initial condition $u(0)=u_0\in\mathcal{M}$ implies $G_l(u(t))\equiv0$, $l=1,2,\ldots,m$, i.e., $u(t)\in\mathcal{M}$. Thus, the dynamics \eqref{eq:idealizedCGAD} preserves the constraints \eqref{eq:m-constraints}. Moreover, we have the following locally exponential convergence result of the dynamics \eqref{eq:idealizedCGAD} around a nondegenerate index-$k$ constrained saddle point.

\begin{theorem}\label{thm:icgad-pg-diminishing}
Assume that $E,G_l\in C^3$, $l=1,2,\ldots,m$, and the constraints \eqref{eq:m-constraints} are regular. Let $u=u(t)$ be the solution of the dynamics \eqref{eq:idealizedCGAD}. Then
\begin{equation}\label{eq:icgad-pg-diminishing}
  \frac{\mathrm{d}}{\mathrm{d} t}\|F(u)\|^2 \leq -2\min\{-\lambda_k(u),\lambda_{k+1}(u)\}\|F(u)\|^2,\quad \forall t\geq0,
\end{equation}
where $\lambda_{k+1}(u)$ is the $(k+1)$-th smallest eigenvalue of $\hat{H}(u)$. Further, if there exists a constant $c>0$ such that $\lambda_{k}(u)\leq-c$ and $\lambda_{k+1}(u)\geq c$ for all $t\geq0$, then
\begin{equation}\label{eq:icgad-pg-diminishing-exp}
  \|F(u(t))\|\leq \mathrm{e}^{-ct}\|F(u_0)\|.
\end{equation}
\end{theorem}

\begin{proof}
Applying Lemma~\ref{lem:dmudF}, and noting that $F(u),v_i(u)\in T_u\mathcal{M}$, $\hat{H}(u)v_i(u)=\lambda_i(u)v_i(u)$, $i=1,2,\ldots,k$, we have
\begin{align*}
\frac{\mathrm{d}}{\mathrm{d} t}\|F(u)\|^2
  &= 2\big\langle F'(u)\dot{u},F(u)\big\rangle \\
  &= 2\big\langle \hat{H}(u)\dot{u}, F(u)\big\rangle-2\sum_{i=1}^m\sum_{j=1}^m g_{ij}(u) \big\langle G_j''(u)F(u),\dot{u}\big\rangle \big\langle G_i'(u),F(u)\big\rangle  \\
  &= 2\bigg\langle \hat{H}(u)\bigg(- F(u) + 2\sum_{i=1}^k\langle F(u),v_i(u)\rangle v_i(u)\bigg), F(u)\bigg\rangle \\
  &= -2\bigg\langle\bigg(\hat{H}(u)-2\sum_{i=1}^k\lambda_i(u)\big[v_i(u)\otimes v_i(u)\big]\bigg)F(u), F(u)\bigg\rangle.
\end{align*}
Clearly, the smallest eigenvalue of the linear operator $A(u):T_{u}\mathcal{M}\to T_{u}\mathcal{M}$, with
\[ A(u):=\hat{H}(u)-2\sum_{i=1}^k\lambda_i(u)\big[v_i(u)\otimes v_i(u)\big],\]
is $c_u:=\min\{-\lambda_k(u), \lambda_{k+1}(u)\}$. Thus, $\langle A(u)\varphi, \varphi\rangle\geq c_{u}\|\varphi\|^2$ for all $\varphi\in T_{u}\mathcal{M}$, and
\[ \frac{\mathrm{d}}{\mathrm{d} t}\|F(u)\|^2=-2\langle A(u)F(u), F(u)\rangle\leq -2c_{u}\|F(u)\|^2. \] 
This is \eqref{eq:icgad-pg-diminishing}. Further, when $c_u\geq c$ for some constant $c>0$ and for all $t\geq0$, \eqref{eq:icgad-pg-diminishing} becomes $\frac{\mathrm{d}}{\mathrm{d} t}\|F(u)\|^2\leq -2c\|F(u)\|^2$. Then the conclusion \eqref{eq:icgad-pg-diminishing-exp} follows from the Gr\"{o}nwall's inequality.
\end{proof}

\begin{remark}
Under all assumptions of Theorem~\ref{thm:icgad-pg-diminishing}, if some additional assumptions on compactness (e.g., the constrained Palais--Smale condition \cite{VGOLMM2}) are made, one can establish the existence of a nondegenerate index-$k$ constrained saddle point $u_*\in\mathcal{M}$ such that, for any initial data $u_0\in\mathcal{M}$ near $u_*$, the solution $u=u(t)$ of the dynamics \eqref{eq:idealizedCGAD} converges to $u_*$ as $t\to+\infty$ with exponential convergence rate:
\begin{align*}
\|u(t)-u_*\| 
&\leq \int_{t}^{\infty} \|\dot{u}(s)\|\, \mathrm{d}s  \\
&= \int_{t}^{\infty}\bigg\|F(u(s))-2\sum_{i=1}^k\langle F(u(s)),v_i(u(s))\rangle v_i(u(s))\bigg\| \mathrm{d}s \\
&= \int_{t}^{\infty}\|F(u(s))\|\, \mathrm{d}s \\
&\leq \mathrm{e}^{-ct}\|F(u_0)\|/c.
\end{align*}
\end{remark}

\section{Applications to finding excited states of single-component BECs}\label{sec:esbec}

The CGAD can be applied to solve many scientific problems. In this section, we apply the CGAD \eqref{eq:cgad-idxk} to find real-valued excited states of single-component BECs.

Within the mean-field theory, the GP energy functional of the wave function $\phi=\phi(\mathbf{x})$ of a single-component BEC in $d$ ($d=1,2,3$) dimension is given as \cite{DGPS1999RMP,BC2013KRM}
\begin{equation}\label{eq:esbec-GPenergy}
  E(\phi)=\int_{U}\left(\frac12|\nabla\phi|^2+V(\mathbf{x})|\phi|^2+\frac{\beta}{2}|\phi|^4\right)\mathrm{d}\mathbf{x},
\end{equation}
where $U\subset\mathbb{R}^d$ is the spatial domain, $V(\mathbf{x})\geq0$ is the real-valued trapping potential and the parameter $\beta\in\mathbb{R}$ characterizes the strength of the interaction. When $U$ is bounded, the homogeneous Dirichlet boundary conditions (i.e., $\phi|_{\partial U}=0$) can be imposed. In the following, we assume that all wave functions involved below are real-valued functions for simplicity.

The stationary state of a BEC is usually defined as the eigenfunction $\phi$ to the Euler--Lagrange equation (or time-independent GPE) \cite{BC2013KRM}
\begin{align}
  -\frac12\Delta\phi+V(\mathbf{x})\phi+\beta|\phi|^2\phi =\mu\phi, \label{eq:esbec-tigpe1}
\end{align}
under the normalization constraint
\begin{align}
  \|\phi\|^2:=\int_{U}|\phi(\mathbf{x})|^2\mathrm{d}\mathbf{x} =1, \label{eq:esbec-tigpe2}
\end{align}
with $\mu$ the corresponding eigenvalue or chemical potential. When $\phi$ is an eigenfunction of \eqref{eq:esbec-tigpe1}-\eqref{eq:esbec-tigpe2}, the corresponding chemical potential is given as
\begin{equation}\label{eq:esbec-muphi}
  \mu(\phi)=\int_{U}\left(\frac12|\nabla\phi|^2+V(\mathbf{x})|\phi|^2+\beta|\phi|^4\right)\mathrm{d}\mathbf{x}
  =E(\phi)+\frac{\beta}{2}\int_{U}|\phi|^4\mathrm{d}\mathbf{x}.
\end{equation}
The ground state is a stationary state with the lowest value of GP energy functional $E$, while stationary states with higher energies are called excited states \cite{BC2013KRM}. Noticing that $E'(\phi)=-\Delta\phi+2V(\mathbf{x})\phi+2\beta|\phi|^2\phi$ and setting $G(\phi)=\|\phi\|^2-1$, we have $G'(\phi)=2\phi$, and \eqref{eq:esbec-tigpe1}-\eqref{eq:esbec-tigpe2} turns to be
\begin{equation}
  E'(\phi)=\mu G'(\phi),\quad G(\phi)=0.
\end{equation}
Thus, all eigenfunctions of \eqref{eq:esbec-tigpe1}-\eqref{eq:esbec-tigpe2} are exactly the constrained critical points of the GP energy functional $E$ \eqref{eq:esbec-GPenergy} on the unit spherical manifold $S=\{\phi\in L^2(U): \phi|_{\partial U}=0, G(\phi)=\|\phi\|^2-1=0,E(\phi)<\infty\}$. The ground state is the constrained minimizer of $E$ \eqref{eq:esbec-GPenergy} on $S$. Since constrained saddle points possess higher energy than that of the ground state, they are sure to be excited states. Although there may be excited states that are not constrained saddle points, such as constrained local minima with higher energies than that of the ground state, here we only consider the excited states corresponding to constrained saddle points of the GP energy functional $E$.

Taking $\langle\cdot,\cdot\rangle$ as the real $L^2$ inner product (or duality pairing), the tangent space of the constraint manifold $S$ at $\phi\in S$ is $T_{\phi}S=\{v\in L^2(U):\langle G'(\phi),v\rangle=0\}=\spn\{\phi\}^\bot$. For $\phi\in S$, the orthogonal projection operator from $L^2(U)$ onto $T_{\phi}S$ is $P_{\phi}=I-(\phi\otimes\phi)$. Then, the projected gradient of the energy functional $E$ at $\phi$ reads as
\begin{align} 
F(\phi):=P_{\phi}E'(\phi)=E'(\phi)-\mu(\phi)G'(\phi)=-\Delta\phi+2V(\mathbf{x})\phi+2\beta|\phi|^2\phi-2\mu(\phi)\phi, 
\end{align}
with $\mu(\phi)=\frac12\langle E'(\phi),\phi\rangle$ given in \eqref{eq:esbec-muphi}. The effective and projected Hessian operators at $\phi\in S$ are, respectively, given as $H(\phi) =E''(\phi)-\mu(\phi)G''(\phi)$ and $\hat{H}(\phi) =P_{\phi}H(\phi)P_{\phi}$ where $E''(\phi) = -\Delta+2\left(V(\mathbf{x})+3\beta|\phi|^2\right)I$ and $G''(\phi)=2I$ with $I$ the identity operator.

We remark that any constrained saddle point must be an excited state, thus we call an index-$k$ constrained saddle point $\phi_k$ an \emph{index-$k$ excited state}. Now one can distinguish the ground state $\phi_g\in S$ and different excited states according to their energies, chemical potentials (i.e., eigenvalues), and Morse indices. A very interesting question is whether the index-$k$ excited state $\phi_k$ is precisely the \emph{$k$-th excited state} in the sense that
\begin{equation}\label{eq:energy-order}
  E(\phi_g)<E(\phi_1)<E(\phi_2)<\cdots<E(\phi_k)<\cdots,
\end{equation}
and/or whether it is the \emph{$k$-th eigenstate} such that
\begin{equation}\label{eq:chemical-order}
  \mu(\phi_g)<\mu(\phi_1)<\mu(\phi_2)<\cdots<\mu(\phi_k)<\cdots.
\end{equation}

\subsection{Properties of excited states in linear case}\label{sec:esbec-b0}

For the linear case (i.e., $\beta=0$), the nonlinear eigenvalue problem \eqref{eq:esbec-tigpe1}-\eqref{eq:esbec-tigpe2} reduces to
\begin{equation}\label{eq:neigp:beta=0}
  -\frac12\Delta\phi+V(\mathbf{x})\phi=\mu\phi,\quad G(\phi)=\|\phi\|^2-1=0,\quad \mathbf{x}\in U
\end{equation}
with homogeneous Dirichlet boundary conditions $\phi|_{\partial U}=0$, and the energy \eqref{eq:esbec-GPenergy} and the chemical potential \eqref{eq:esbec-muphi} are identical. As a result, \eqref{eq:energy-order} and \eqref{eq:chemical-order} are completely equivalent.

The following result provides an exact characterization for all excited states in linear case.

\begin{theorem}\label{thm:esbec-b0}
Assume that $U$ is a bounded domain with Lipschitz boundary, $0\leq V(\mathbf{x})\in L^{\infty}(U)$, $\beta=0$, and $\phi_*\in S$ is an eigenfunction of the linear eigenproblem \eqref{eq:neigp:beta=0} with $\mu_*=\mu(\phi_*)=E(\phi_*)$ the corresponding eigenvalue. Let $0<\mu_0<\mu_1\leq\mu_2\leq\cdots$ be all the eigenvalues with $\phi_0,\phi_1,\phi_2,\cdots$ the corresponding orthonormal eigenfunctions of the linear eigenproblem \eqref{eq:neigp:beta=0}. Then $\phi_*$ is an index-$k$ ($k=1,2,\ldots$) excited state if and only if $\mu_{k-1}<\mu_*=\mu_k$. Moreover, the unstable tangent subspace of an index-$k$ excited state $\phi_*$ is $T^-(\phi_*)=\mathrm{span}\{\phi_0,\phi_1,\ldots,\phi_{k-1}\}$.
\end{theorem}

\begin{proof} 
Denote $A:=-\frac12\Delta+V(\mathbf{x})I$. According to the spectral theory of uniformly elliptic operators \cite{GT}, the set of all eigenfunctions of $A$ forms a complete basis of $L^2(U)$. Since $A\phi_*=\mu_*\phi_*$, we have $H(\phi_*)=-\Delta+2(V(\mathbf{x})-\mu_*)I=2(A-\mu_*I)$, and therefore, $H(\phi_*)\phi_*=2(A\phi_*-\mu_*\phi_*)=0$. Then, for any $\xi\in T_{\phi_*}S$, noting that $P_{\phi_*}\xi=\xi$, we have
\begin{align*}
  \hat{H}(\phi_*)\xi
  &= P_{\phi_*}H(\phi_*)\xi 
  = H(\phi_*)\xi-\langle\phi_*,H(\phi_*)\xi\rangle\phi_* \\
  &= H(\phi_*)\xi-\langle H(\phi_*)\phi_*,\xi\rangle\phi_* 
  = H(\phi_*)\xi
  =2(A-\mu_*I)\xi, 
\end{align*}
Thus, $\hat{H}(\phi_*)=2(A-\mu_*I):T_{\phi_*}S\to T_{\phi_*}S$.

{\em Necessity.} Suppose that $\phi_*$ is an index-$k$ excited state. Then the linear operator $\hat{H}(\phi_*)$ has exactly $k$ negative eigenvalues. Let $\lambda_0\leq\lambda_1\leq\cdots\leq\lambda_{k-1}<0\leq\lambda_k\leq\lambda_{k+1}\leq\cdots$ be all eigenvalues of $\hat{H}(\phi_*)$ with $\{\eta_j\}_{j=0}^{\infty}\subset T_{\phi_*}S$ the corresponding orthonormal eigenfunctions. Then $L^2(U)=\spn\{\phi_*\}\oplus T_{\phi_*}S=\spn\{\phi_*,\eta_j,j=0,1,\ldots\}$. Noting that $A\phi_*=\mu_*\phi_*$ and
\[ A\eta_j=\left(\mu_*I+\frac12\hat{H}(\phi_*)\right)\eta_j =\left(\mu_*+\frac{\lambda_j}{2}\right)\eta_j,\quad j=0,1,\ldots, \]
one obtains that all eigenvalues of $A$ are
\[ \mu_*+\frac{\lambda_0}{2}\leq \mu_*+\frac{\lambda_1}{2}\leq\cdots\leq\mu_*+\frac{\lambda_{k-1}}{2}
<\mu_*\leq \mu_*+\frac{\lambda_{k}}{2}\leq\mu_*+\frac{\lambda_{k+1}}{2}\leq\cdots. \]
Therefore, $\mu_{k-1}=\mu_*+\frac{\lambda_{k-1}}{2}<\mu_*=\mu_k$.

{\em Sufficiency.} 
Suppose that $\mu_{k-1}<\mu_*=\mu_k$. Without loss of generality, assume $\phi_*=\phi_k$. Then $T_{\phi_*}S=\spn\{\phi_*\}^\bot=\spn\{\phi_0,\phi_1,\ldots,\phi_{k-1},\phi_{k+1},\ldots\}$. Note that
\begin{align*}
\hat{H}(\phi_*)\phi_i=2(A-\mu_kI)\phi_i=2(\mu_i-\mu_k)\phi_i,\quad i=0,1,\ldots,k-1,\ldots,k+1,\ldots.
\end{align*}
The maximum negative definite subspace of the linear operator $\hat{H}(\phi_*):T_{\phi_*}S\to T_{\phi_*}S$ is given as $T^-= \spn\{\phi_0,\phi_1,\ldots,\phi_{k-1}\}$. Thus, the Morse index of $\phi_*$ is $\dim(T^-)=k$, i.e., $\phi_*$ is an index-$k$ excited state.
\end{proof}

\begin{remark}
The result of Theorem~\ref{thm:esbec-b0} can be extended to the cases when $U=\mathbb{R}^d$ and the potential $V(\mathbf{x})$ satisfies: $V(\mathbf{x})$ is continuous in $\mathbb{R}^d$, $V(\mathbf{x})\geq0$ and $\lim_{|\mathbf{x}|\to\infty}V(\mathbf{x})=\infty$. In addition, Theorem~\ref{thm:esbec-b0} can also be proved by applying the generalized Courant-Fischer formula or min-max principle for self-adjoint operators (see, e.g., \cite{LiebLoss2001}).
\end{remark}

Noting that, when $\beta=0$, $E(\phi)=\mu(\phi)$, we have the following corollaries.

\begin{corollary}\label{cor:esbec-b0}
Under assumptions of Theorem~\ref{thm:esbec-b0}, if further $\mu_0<\mu_1<\cdots<\mu_k$, i.e., $\mu_0,\mu_1,\ldots,\mu_{k-1}$ are single-fold eigenvalues, then $\phi_*$ is an index-$k$ excited state if and only if it is the $k$-th eigenstate defined in \eqref{eq:chemical-order} (equivalently, it is the $k$-th excited state defined in \eqref{eq:energy-order}). In particular, since $\mu_0$ is single-fold, the index-1 excited state is exactly the first excited state and the first eigenstate. 
\end{corollary}

\begin{corollary}\label{cor:esbec-b0-saddle}
Under assumptions of Theorem~\ref{thm:esbec-b0}, all excited states are constrained saddle points, and the ground state $\phi_g$ (up to the sign) is the only possible constrained local minimizer and thus the constrained global minimizer.
\end{corollary}

\begin{example}\label{ex:esbec-b0-box}
Assume $\beta=0$. Take $V(\mathbf{x})$ as the box potential:
\[
  V_{\mathrm{box},L}(\mathbf{x})=
  \begin{cases}
    0, & \mathbf{x}\in U:=[0,L]^d, \\
    \infty, & \mathbf{x}\notin U,
  \end{cases} \quad  \mathbf{x}=(x_1,x_2,\cdots,x_d)^T
\]
with $L>0$ the width of the box. The eigenpairs of the linear eigenproblem \eqref{eq:neigp:beta=0} are
\begin{equation}\label{eq:b0eigpairs-box}
  \phi^{\mathrm{box}}_{\mathbf{j}}(\mathbf{x})=\left(\frac2L\right)^{d/2} \prod_{\alpha=1}^{d} \sin\frac{(j_\alpha+1)\pi x_\alpha}{L}, 
  \quad\mathbf{x}\in U,\quad 
  \mu^{\mathrm{box}}_{\mathbf{j}} = \frac{\pi^2}{2L^2}\sum_{\alpha=1}^d (j_\alpha+1)^2,
\end{equation}
for all $\mathbf{j}=(j_1,j_2,\cdots,j_d)\in\mathbb{N}^d$. 
From Theorem~\ref{thm:esbec-b0} and Corollary~\ref{cor:esbec-b0}, we have the following conclusions:
\begin{enumerate}[(i)]
  \item $\phi^{\mathrm{box}}_{\mathbf{0}}$ is the ground state.
  \item When $d=1$, we have $\mu^{\mathrm{box}}_j=\frac{\pi^2}{2L^2}(j+1)^2$, $j=0,1,\cdots$ and therefore $\mu^{\mathrm{box}}_0<\mu^{\mathrm{box}}_1<\mu^{\mathrm{box}}_2<\mu^{\mathrm{box}}_3<\cdots$. Thus $\phi^{\mathrm{box}}_k$ ($k\geq1$) is exactly an index-$k$ excited state as well as the $k$-th excited state (and the $k$-th eigenstate) with its unstable tangent subspace $T^-(\phi^{\mathrm{box}}_k)=\spn\{\phi^{\mathrm{box}}_j:j=0,1,\cdots,k-1\}$.
  \item When $d=2$, the first few stationary states with corresponding energy levels and Morse indices are listed in Table~\ref{tab:eis-linearbox}. Thus any $\phi\in S\cap\spn\{\phi^{\mathrm{box}}_{(1,0)},\phi^{\mathrm{box}}_{(0,1)}\}$ is an index-1 excited state as well as the first excited state with its unstable tangent subspace spanned by $\phi^{\mathrm{box}}_{(0,0)}$. However, the second excited state $\phi^{\mathrm{box}}_{(1,1)}$ is actually an index-3 excited state with its unstable tangent subspace spanned by $\phi^{\mathrm{box}}_{(0,0)}$, $\phi^{\mathrm{box}}_{(1,0)}$ and $\phi^{\mathrm{box}}_{(0,1)}$. In general, as shown in Table~\ref{tab:eis-linearbox}, the order of energies or chemical potentials of excited states is accordance with that of Morse indices.
\end{enumerate}
\end{example}

\begin{table}[!ht]\footnotesize
\centering
\caption{Energy levels and indices of the first few excited states for the box potential in 2D with $\beta=0$.}
\label{tab:eis-linearbox}
\begin{tabular*}{\textwidth}{@{\extracolsep{\fill}}cllllllll}
  \hline
  $\mathbf{j}$ & $(0,0)$ & $(1,0),(0,1)$ & $(1,1)$ & $(2,0),(0,2)$ & $(2,1),(1,2)$ & $(3,0),(0,3)$ & $(2,2)$ \\  \hline
  $(2L^2/\pi^2)\mu^{\mathrm{box}}_{\mathbf{j}}$ & 2 & 5 & 8 & 10 & 13 & 17 & 18  \\ 
  energy levels & 0 & 1 & 2 & 3 & 4 & 5 & 6 \\  
  indices & 0 & 1 & 3 & 4 & 6 & 8 & 10  \\
  \hline
\end{tabular*}
\end{table}

\begin{example}\label{ex:esbec-b0-ho}
Assume $\beta=0$. Take $V(\mathbf{x})$ as the harmonic oscillator potential:
\[
  V_{\mathrm{ho}}(\mathbf{x})= \frac12|\mathbf{x}|^2= \frac12\sum_{\alpha=1}^d x_\alpha^2,\quad \mathbf{x}=(x_1,x_2,\cdots,x_d)^T\in U=\mathbb{R}^d.
\]
Then the eigenpairs of the linear eigenproblem \eqref{eq:neigp:beta=0} are given as
\begin{equation}\label{eq:b0eigpairs-ho}
  \phi^{\mathrm{ho}}_{\mathbf{j}}(\mathbf{x}) = \prod_{\alpha=1}^{d}\hat{h}_{j_\alpha}(x_\alpha), \quad
  \mu^{\mathrm{ho}}_{\mathbf{j}} = |\mathbf{j}|+\frac{d}{2},\quad
  \mathbf{j}=(j_1,j_2,\cdots,j_d)\in\mathbb{N}^d,
\end{equation}
where $|\mathbf{j}|:=\sum_{\alpha=1}^{d}j_{\alpha}$, $\hat{h}_j(x)$ are the Hermite functions:
\begin{equation}\label{eq:hermitefuns}
  \hat{h}_j(x)=\frac{1}{\pi^{1/4}\sqrt{2^j j!}}\mathrm{e}^{-x^2/2}h_j(x), \quad j=0,1,2,\cdots,
\end{equation}
with $h_j(x)=(-1)^j\mathrm{e}^{x^2}\frac{\mathrm{d}^j}{\mathrm{d}x^j}(\mathrm{e}^{-x^2})$ the Hermite polynomials. Obviously, $\phi^{\mathrm{ho}}_{\mathbf{0}}$ is the ground state. From Theorem~\ref{thm:esbec-b0} and Corollary~\ref{cor:esbec-b0}, any function $\phi\in S\cap\spn\{\phi^{\mathrm{ho}}_{\mathbf{j}}:\mathbf{j}\in\mathbb{N}^d,|\mathbf{j}|=k\}$ is the $k$-th excited state and an index-$i_d(k)$ excited state with its unstable tangent subspace $T^-(\phi)=\spn\{\phi^{\mathrm{ho}}_{\mathbf{j}}:\mathbf{j}\in\mathbb{N}^d,|\mathbf{j}|\leq k-1\}$, where
\[
  i_d(k)=\#\{\mathbf{j}\in\mathbb{N}^d:|\mathbf{j}|\leq k-1\}=
  \begin{cases}
    k, &  d=1, \\
    \frac12k(k+1), &  d=2, \\
    \frac16k(k+1)(k+2), & d=3.
  \end{cases}
\]
It is observed that, $i_d(k)=k$ if either $d=1$ or $k=0,1$; Otherwise, $i_d(k)>k$.
\end{example}

\subsection{CGAD for single-component BECs and its time discretization}

We now propose the formulation of the CGAD for computing excited states of a single-component BEC and its efficient time discretization scheme.

Let $\phi\in S$ be an approximation of an index-$k$ excited state and $\{v_i\}_{i=1}^k\subset T_{\phi}S$ be the approximations of corresponding unstable tangent directions. Noting that $G'(\phi)=2\phi$ and $G''(\phi)=2I$, by applying the CGAD \eqref{eq:cgad-idxk} to the single-component BEC model, we obtain
\begin{equation}\label{eq:cgad-idxk-bec}
  \left\{
  \begin{aligned}
  \gamma_0\partial_t\phi(\mathbf{x},t) &= -F(\phi) + 2\sum_{j=1}^k \langle F(\phi),v_j\rangle v_j,\\
  \gamma_i\partial_t v_i(\mathbf{x},t) &= -\hat{H}(\phi)v_i+\sum_{j=1}^i \lambda_{ij} v_j +2\bar{\lambda}_i \phi,\quad i=1,2,\ldots,k,
  \end{aligned}\right.
\end{equation}
where $\gamma_i>0$ ($i=0,1,\ldots,k$) are relaxation parameters, Lagrange multipliers $\lambda_{ij}$ and $\bar{\lambda}_{i}$ are given as
\begin{align*}
  \lambda_{ij} = \left(1+\frac{\gamma_i}{\gamma_j}-\delta_{ij}\right)\big\langle \hat{H}(\phi) v_i,v_j\big\rangle, \quad
  \bar{\lambda}_{i} = \frac{\gamma_i}{2\gamma_0\|\phi\|^2}\bigg\langle v_i, F(\phi) - 2\sum_{j=1}^k\langle F(\phi),v_j\rangle v_j\bigg\rangle, \quad 1\leq j\leq i\leq k.
\end{align*}
Lemma~\ref{lem:cgad-idxk-conpty} states that \eqref{eq:cgad-idxk-bec} preserves constraints $\phi\in S$, $v_i\in T_{\phi}S$, $\langle v_i,v_j\rangle=\delta_{ij}$, i.e.,
\begin{equation}\label{eq:cgad-idxk-bec-cons}
  \|\phi\|^2=1,\quad \langle \phi,v_i\rangle=0,\quad \langle v_i,v_j\rangle=\delta_{ij},\quad 1\leq j\leq i\leq k.
\end{equation}
Using \eqref{eq:cgad-idxk-bec-cons}, we have $\langle F(\phi),v_i\rangle= \langle E'(\phi),P_{\phi}v_i\rangle = \langle E'(\phi),v_i\rangle$, 
\begin{align*}
&\hat{H}(\phi)v_i = P_{\phi}(E''(\phi)-2\mu(\phi)I)P_{\phi}v_i 
= E''(\phi)v_i-\langle E''(\phi)\phi,v_i\rangle\phi-2\mu(\phi)v_i, \\
&\lambda_{ij}=\left(1+\frac{\gamma_i}{\gamma_j}-\delta_{ij}\right)\langle E''(\phi)v_i,v_j\rangle-2\mu(\phi)\delta_{ij}, \qquad
\bar{\lambda}_i = -\frac{\gamma_i}{2\gamma_0}\langle E'(\phi),v_i\rangle.
\end{align*}
Noting that $E'(\phi)=-\Delta\phi+2V(\mathbf{x})\phi+2\beta|\phi|^2\phi$, $E''(\phi) = -\Delta+2\left(V(\mathbf{x})+3\beta|\phi|^2\right)I$, by taking $\gamma_0=\gamma_1=\cdots=\gamma_k=2$, \eqref{eq:cgad-idxk-bec} can be simplified as
\begin{equation}\label{eq:cgad-esbec}
  \left\{
  \begin{aligned}
  \partial_t\phi(\mathbf{x},t) &= \frac{1}{2}\Delta\phi-V(\mathbf{x})\phi-\beta|\phi|^2\phi+\mu(\phi)\phi + 2\sum_{j=1}^k \xi_j v_j,\\
  \partial_tv_i(\mathbf{x},t) & = \frac{1}{2}\Delta v_i-V(\mathbf{x}) v_i-3\beta|\phi|^2v_i +\sigma_{i}\phi + \sum_{j=1}^i\nu_{ij}v_j,\quad i=1,2\ldots,k,
  \end{aligned}\right.
\end{equation}
where
\begin{align}
\xi_i &= \xi_i(\phi,v_i)
  = \int_{U}\left(\frac12\nabla\phi\cdot\nabla v_i+V(\mathbf{x})\phi v_i+\beta|\phi|^2\phi v_i\right)\mathrm{d}\mathbf{x}, 
  \label{eq:cgad-esbec-xi} \\
\nu_{ij} &= \nu_{ij}(\phi,v_i,v_j)
  = (2-\delta_{ij})\int_{U}\left(\frac12\nabla v_i\cdot\nabla v_j+V(\mathbf{x})v_iv_j+3\beta|\phi|^2v_iv_j\right)\mathrm{d}\mathbf{x}, 
  \label{eq:cgad-esbec-nu} \\
\sigma_{i} & =\sigma_{i}(\phi,v_i)
  = 2\beta\int_{U} |\phi|^2\phi v_i\mathrm{d}\mathbf{x},\quad 1\leq j\leq i\leq k.
  \label{eq:cgad-esbec-dl}
\end{align}

Various suitable numerical schemes could be used to solve \eqref{eq:cgad-esbec}. For simplicity and efficiency, we use the prediction-correction strategy to discretize \eqref{eq:cgad-esbec} in time with a (semi-implicit) backward-forward Euler scheme followed by the Gram--Schmidt orthonormalization process to preserve the constraints \eqref{eq:cgad-idxk-bec-cons} in the discretized level.

The initial data $(\phi^0,v_1^0,\ldots,v_k^0)$ is chosen satisfying the constraints \eqref{eq:cgad-idxk-bec-cons}. Set $t_n=n\tau$, $n=0,1,\ldots$, with $\tau>0$ a selected time step length. Let $(\phi^n,v_1^n,\ldots,v_k^n)$ be the numerical approximation of the solution of \eqref{eq:cgad-esbec} at $t=t_n$. We adopt the following iterative scheme to compute $(\phi^{n+1},v_1^{n+1},\ldots,v_k^{n+1})$ from $(\phi^n,v_1^n,\ldots,v_k^n)$:
\begin{equation}\label{eq:cgad-esbec-semi}
  \left\{
  \begin{aligned}
  &\frac{\tilde{\phi}^{n+1}-\phi^n}{\tau} = \frac{1}{2}\Delta\tilde{\phi}^{n+1}-V(\mathbf{x})\phi^n-\beta|\phi^n|^2\phi^n+\mu^n\phi^n + 2\sum_{j=1}^k \xi_j^n v_j^n,\\
  &\,\frac{\tilde{v}_i^{n+1}-v_i^n}{\tau} = \frac{1}{2}\Delta\tilde{v}_i^{n+1}-V(\mathbf{x}) v_i^n-3\beta|\phi^n|^2v_i^n +\sigma_{i}^n\phi^n + \sum_{j=1}^i\nu_{ij}^nv_j^n,\quad  i=1,2,\ldots,k,\\
  & \big[\phi^{n+1},v_1^{n+1},\ldots,v_k^{n+1}\big] = \mathrm{GSON}\left(\big[\tilde{\phi}^{n+1},\tilde{v}_1^{n+1},\ldots,\tilde{v}_k^{n+1}\big]\right),
  \end{aligned}\right.
\end{equation}
where $\mu^n=\mu(\phi^n)$ \eqref{eq:esbec-muphi}, $\xi_i^n = \xi_i(\phi^n,v_i^n)$ \eqref{eq:cgad-esbec-xi}, $\nu_{ij}^n = \nu_{ij}(\phi^n,v_i^n,v_j^n)$ \eqref{eq:cgad-esbec-nu}, $\sigma_{i}^n =\sigma_{i}(\phi^n,v_i^n)$ \eqref{eq:cgad-esbec-dl}, and $\mathrm{GSON}$ denotes the standard Gram--Schmidt orthonormalization procedure to preserve that $(\phi^{n+1},v_1^{n+1},\ldots,v_k^{n+1})$ satisfies the constraints \eqref{eq:cgad-idxk-bec-cons}. We remark that $\mathrm{GSON}$ in \eqref{eq:cgad-esbec-semi} can also be implemented with its variants (e.g., the modified Gram-Schmidt algorithm or the Gram-Schmidt with re-orthogonalization) to overcome the numerical instability (of round-off errors) that may occur in some extreme and ill-conditioned cases. We choose the current version of $\mathrm{GSON}$ (i.e., the standard Gram--Schmidt procedure) in our numerical experiments for simplicity since it works well for all cases of this paper.

Clearly, the main computational cost of the scheme \eqref{eq:cgad-esbec-semi} at each time step is to solve a completely decoupled system of $k+1$ linear elliptic equations with constant coefficients. All equations in the system take the same form: $-\frac{\tau}{2}\Delta u+u=f$, only with different right-hand-side terms $f$. Thus, they can be solved very efficiently, especially when a fast Poisson solver (e.g., fast Fourier transform) and parallel algorithms are available.

In our numerical computation, the iterative scheme \eqref{eq:cgad-esbec-semi} for computing index-$k$ excited states of a single-component BEC is stopped when the following stopping criteria are satisfied:
\begin{equation}\label{eq:stopping}
\|F^n\|_{\infty}<\varepsilon, \quad
\frac{\|\phi^{n+1}-\phi^n\|_{\infty}}{\tau}<\varepsilon, \quad
\frac{\|v_i^{n+1}-v_i^n\|_{\infty}}{\tau}<\varepsilon,\quad i=1,2,\ldots,k,
\end{equation}
where $F^n:=-\frac12\Delta\phi^n+V(\mathbf{x})\phi^n+\beta|\phi^n|^2\phi^n-\mu^n\phi^n$ is the residual of the Euler--Lagrange equation \eqref{eq:esbec-tigpe1} at $(\mu^n,\phi^n)$, and $\varepsilon>0$ is a given tolerance.

\begin{remark}
In order to improve the computational efficiency of the scheme \eqref{eq:cgad-esbec-semi}, one can introduce a suitable stabilization term \cite{BD2004SISC,BCL2006JCP,BC2013KRM} with constant coefficient for each equation in \eqref{eq:cgad-esbec-semi} so that the larger step length can be chosen in practice. Our numerical experiments show that such a stabilized version of \eqref{eq:cgad-esbec-semi} is efficient. However, due to the limit of page, we leave the rigorous stability analysis for the scheme \eqref{eq:cgad-esbec-semi} to future work. It is worthwhile to mention that, on the stability at large step size for index-1 saddle points of functionals, one existing approach is to use the iterative minimization formulation (IMF) \cite{GLZ2016JCP} to have a sequence of minimization problems and to design a convex splitting method \cite{GZ2018JCP} to minimize the auxiliary functional at each cycle of the IMF. 
\end{remark}%

\begin{remark}
If one takes $k=0$ (i.e., remove all approximations of the unstable directions $v_i$) in \eqref{eq:cgad-esbec}, then the CGAD \eqref{eq:cgad-esbec} reduces to the continuous normalized gradient flow (CNGF) \cite{BD2004SISC} for computing the ground state of single-component BECs, and the corresponding time discretization scheme \eqref{eq:cgad-esbec-semi} becomes the backward-forward Euler scheme (followed by a normalization step) for the CNGF (see \cite{LC2021SISC}). 
\end{remark}

\subsection{Numerical results}

We now report the numerical results of the excited states of single-component BECs in 1D and 2D computed by the numerical scheme \eqref{eq:cgad-esbec-semi} of the CGAD. In particular, the asymptotic properties of the energies and chemical potentials of excited states corresponding to different parameters $\beta$ are investigated. Meanwhile, the energies and chemical potentials of the ground state and excited states with different Morse indices are compared.

In our experiments, the following three types of potentials are considered \cite{BC2013KRM}: 
\begin{enumerate}[(i)]
\item the box potential
\begin{equation}\label{eq:Vbox}
  V_{\mathrm{box}}(\mathbf{x})=\begin{cases} 0, & \mathbf{x}\in U=[0,1]^d, \\  \infty, & \mathbf{x}\notin U, \end{cases} \qquad d=1,2,
\end{equation}
\item the harmonic oscillator potential
\begin{equation}\label{eq:Vho}
  V_{\mathrm{ho}}(\mathbf{x}) =
  \begin{cases}
    x^2/2, & d=1, \\
    (x^2+y^2)/2, & d=2,
  \end{cases}
\end{equation}
 \item the harmonic oscillator plus optical lattice potential
\begin{equation}\label{eq:Vhol}
  V_{\mathrm{hol}}(\mathbf{x})=V_{\mathrm{ho}}(\mathbf{x})+
  \begin{cases}
    \kappa\sin^2(\pi x/4), & d=1, \\
    \kappa\big(\sin^2(\pi x/4)+\sin^2(\pi y/4)\big), & d=2,
  \end{cases}
\end{equation}
with $\kappa>0$ the depth of the optical lattice.
\end{enumerate}

We compute excited states by the numerical scheme \eqref{eq:cgad-esbec-semi} of CGAD with time step $\tau=0.01$. The stopping criterion \eqref{eq:stopping} with $\varepsilon=10^{-12}$ is applied. For numerical comparison, we also use the normalized gradient flow \cite{BD2004SISC,LC2021SISC} to compute the ground state. All algorithms are implemented on a bounded domain $U\subset\mathbb{R}^d$ ($d=1,2$) with the spatial sine-pseudospectral discretization (see, e.g., \cite{BCL2006JCP}) with mesh size $h=\frac{1}{32}$.

\subsubsection{Numerical results in 1D}

\begin{example}\label{ex:es1d:box:vb}
In this example, the first few excited states with the box potential $V(x)=V_{\mathrm{box}}(x)$ \eqref{eq:Vbox} in 1D and various interaction coefficient $\beta$ are computed. Then, the asymptotic properties of their energies and chemical potentials are studied. 
\end{example}

Let $\phi_k^\beta$ be the numerical index-$k$ excited state for specified $\beta$ computed with the initial data: $\phi^0(x)=\phi^{\mathrm{box}}_k(x)=\sqrt{2}\sin((k+1)\pi x)$ and $v_i^0(x)=\phi^{\mathrm{box}}_{i-1}(x)=\sqrt{2}\sin(i\pi x)$, $x\in[0,1]$, $i=1,2,\ldots,k$. Fig.~\ref{fig:es1d_box_bxxx_idxxx} plots the profiles of $\phi_k^{\beta}(x)$, $k=1,2,\cdots,9$, with different $\beta=100,3200,102400$. The energies and chemical potentials of $\phi_g^\beta$ and $\phi_k^\beta$ ($k=1,2,\cdots,9$) for various $\beta$ are listed in Table~\ref{tab:box:es-mus} (the initial guess for the ground state $\phi_g^{\beta}$ is taken as $\phi^0(x)=\phi^{\mathrm{box}}_0(x)=\sqrt{2}\sin(\pi x)$). Moreover, the asymptotic behaviors of the energies of excited states in the weakly repulsive interaction regime, i.e., $0<\beta\ll 1$, and the strongly repulsive interaction regime, i.e., $\beta\gg 1$, are shown in Fig.~\ref{fig:es1d_box_idxxx_es_asymptotics}.

\begin{figure}[!ht]
  \centering
  \includegraphics[width=.3\textwidth]{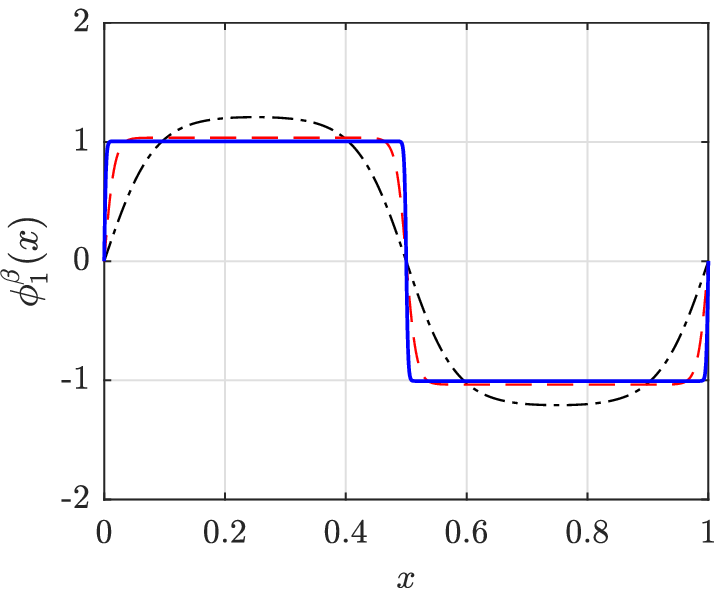}
  \includegraphics[width=.3\textwidth]{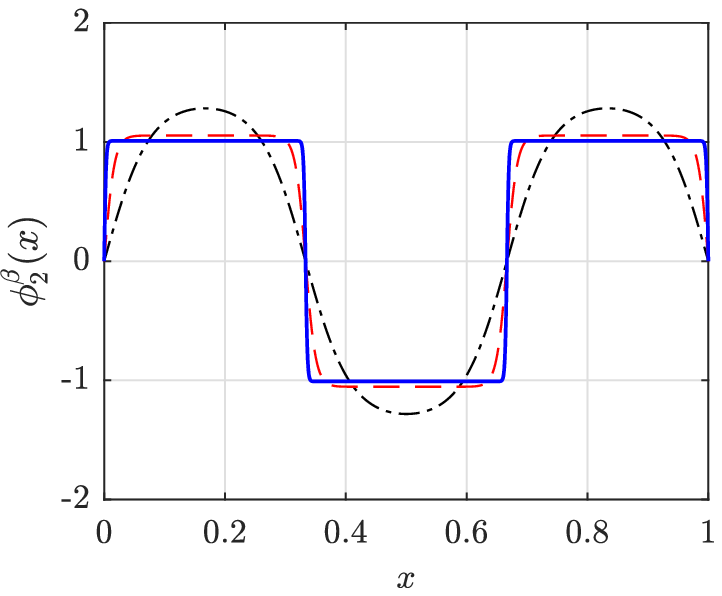}
  \includegraphics[width=.3\textwidth]{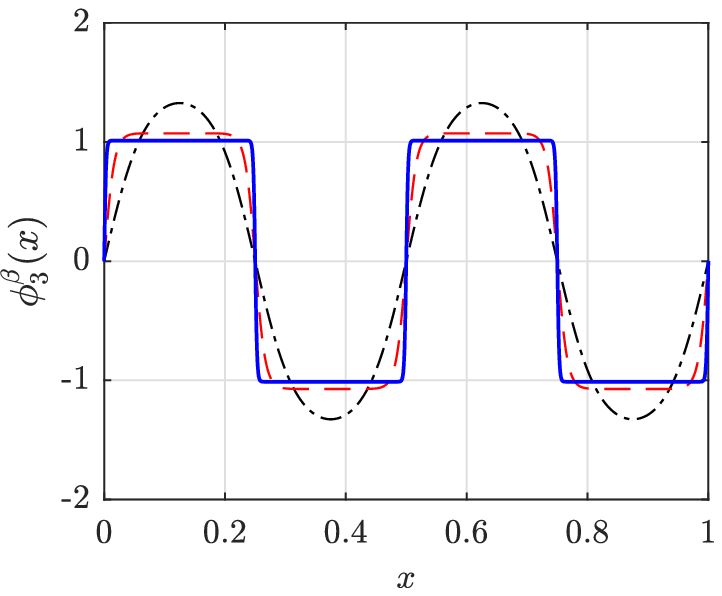} \\
  \includegraphics[width=.3\textwidth]{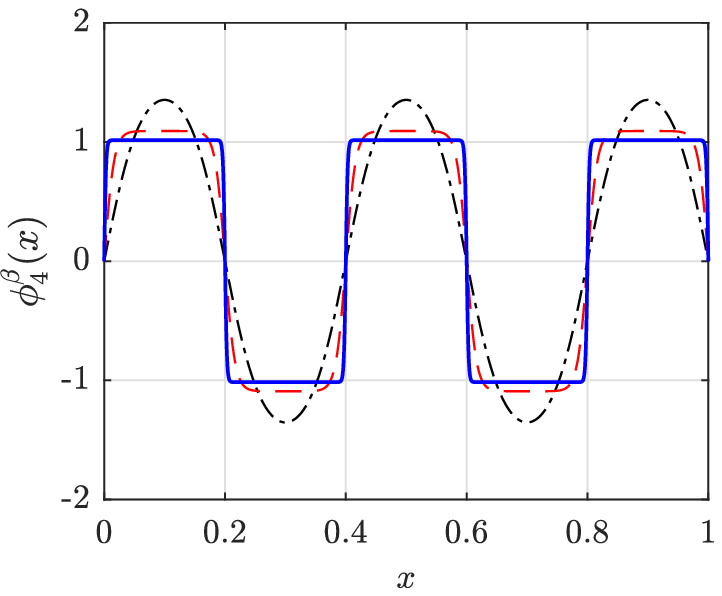}
  \includegraphics[width=.3\textwidth]{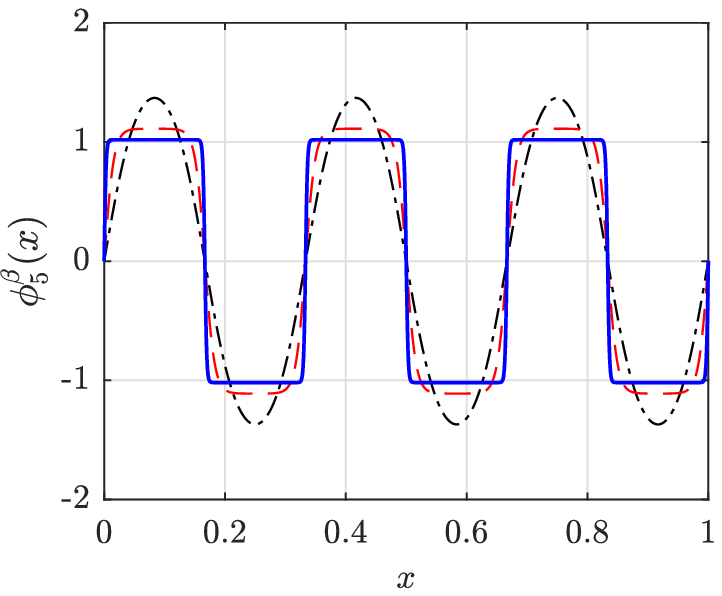}
  \includegraphics[width=.3\textwidth]{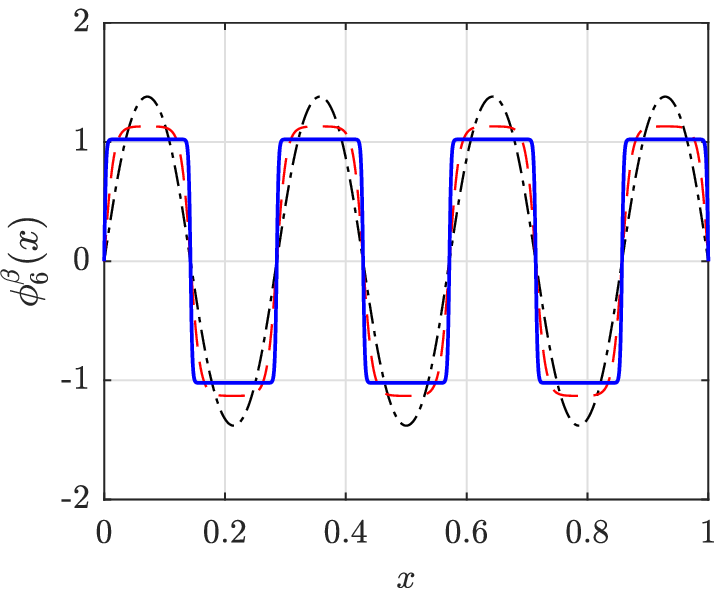} \\
  \includegraphics[width=.3\textwidth]{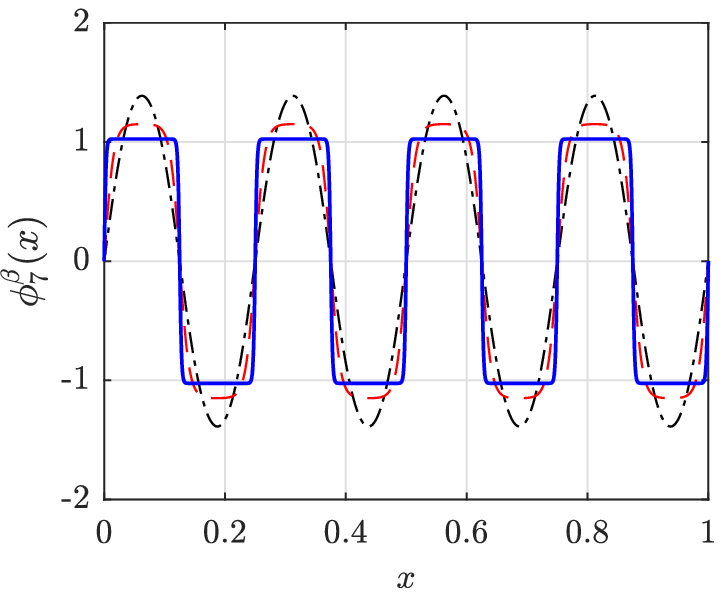}
  \includegraphics[width=.3\textwidth]{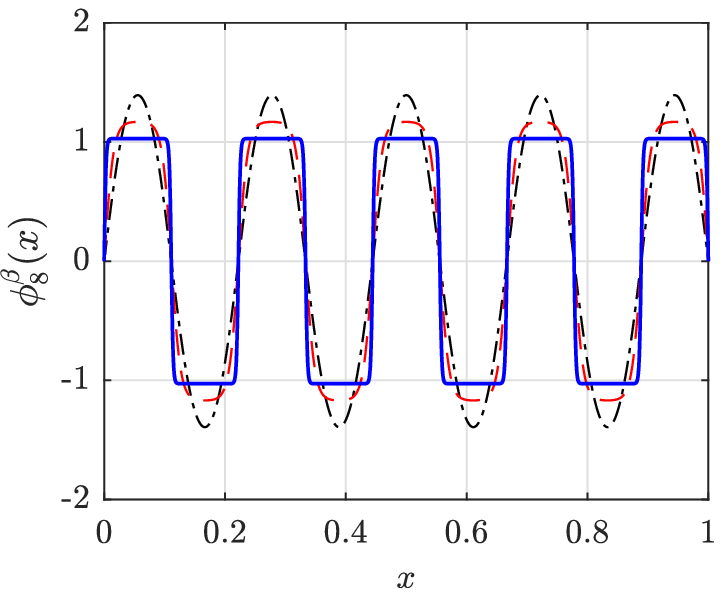}
  \includegraphics[width=.3\textwidth]{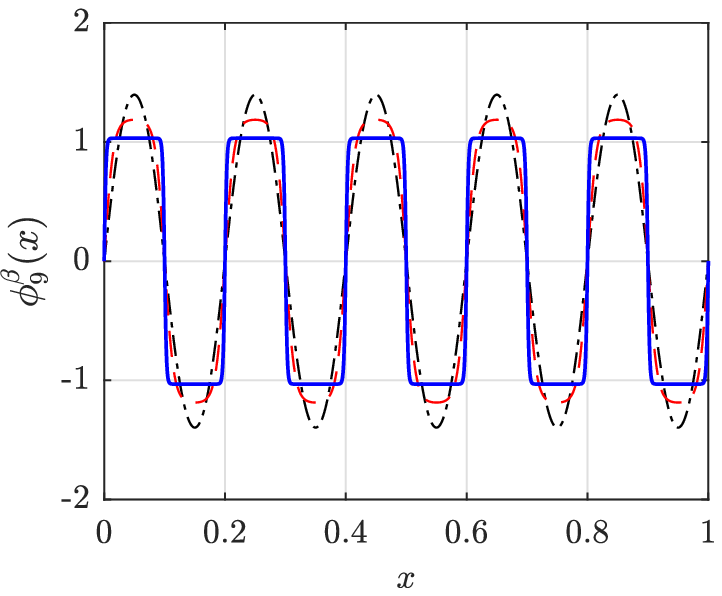}
  \caption{Profiles of index-$k$ ($k=1,2,\ldots,9$) excited states $\phi_k^\beta(x)$ in Example~\ref{ex:es1d:box:vb} with $\beta=100$ (black dash-dot lines), $3200$ (red dash lines) and $102400$ (blue solid lines).}
  \label{fig:es1d_box_bxxx_idxxx}
\end{figure}

\begin{table}[!ht]\footnotesize
  \centering
  \caption{Energies and chemical potentials of the ground state $\phi_g^\beta$ and excited states $\phi_k^\beta$ ($k=1,2,\cdots,9$) versus the interaction coefficient $\beta$ in Example~\ref{ex:es1d:box:vb}.}
  \label{tab:box:es-mus}
\begin{tabular*}{\textwidth}{@{\extracolsep{\fill}}cllllllllll}
\hline
$\beta$ & $E(\phi_g^\beta)$ & $E(\phi_1^\beta)$ & $E(\phi_2^\beta)$ & $E(\phi_3^\beta)$ & $E(\phi_4^\beta)$ & $E(\phi_5^\beta)$ & $E(\phi_6^\beta)$ & $E(\phi_7^\beta)$ & $E(\phi_8^\beta)$ & $E(\phi_9^\beta)$ \\
\hline
0      & 4.93480 & 19.7392 & 44.4132 & 78.9568 & 123.370 & 177.653 & 241.805 & 315.827 & 399.719 & 493.480 \\
0.01   & 4.94230 & 19.7467 & 44.4207 & 78.9643 & 123.378 & 177.660 & 241.813 & 315.835 & 399.726 & 493.488 \\
1      & 5.67870 & 20.4876 & 45.1625 & 79.7064 & 124.120 & 178.403 & 242.555 & 316.577 & 400.469 & 494.230 \\
100    & 65.5472 & 86.4930 & 114.450 & 150.756 & 196.171 & 251.062 & 315.606 & 389.893 & 473.972 & 567.870 \\
1600   & 855.384 & 915.080 & 979.419 & 1048.75 & 1123.46 & 1203.94 & 1290.61 & 1383.89 & 1484.21 & 1591.96 \\
12800  & 6552.87 & 6709.84 & 6871.03 & 7036.56 & 7206.52 & 7381.05 & 7560.27 & 7744.28 & 7933.24 & 8127.25 \\
102400 & 51628.7 & 52061.4 & 52498.2 & 52939.1 & 53384.1 & 53833.4 & 54286.8 & 54744.6 & 55206.6 & 55673.0 \\
\hline
$\beta$ & $\mu(\phi_g^\beta)$ & $\mu(\phi_1^\beta)$ & $\mu(\phi_2^\beta)$ & $\mu(\phi_3^\beta)$ & $\mu(\phi_4^\beta)$ & $\mu(\phi_5^\beta)$ & $\mu(\phi_6^\beta)$ & $\mu(\phi_7^\beta)$ & $\mu(\phi_8^\beta)$ & $\mu(\phi_9^\beta)$ \\
\hline
0      & 4.93480 & 19.7392 & 44.4132 & 78.9568 & 123.370 & 177.653 & 241.805 & 315.827 & 399.719 & 493.480 \\
0.01   & 4.94980 & 19.7542 & 44.4282 & 78.9718 & 123.385 & 177.668 & 241.820 & 315.842 & 399.734 & 493.495 \\
1      & 6.41672 & 21.2345 & 45.9111 & 80.4557 & 124.869 & 179.152 & 243.305 & 317.327 & 401.219 & 494.980 \\
100    & 122.100 & 148.803 & 180.961 & 219.961 & 267.060 & 323.031 & 388.293 & 463.078 & 547.512 & 641.672 \\
1600   & 1682.02 & 1768.20 & 1858.67 & 1953.60 & 2053.11 & 2157.38 & 2266.55 & 2380.85 & 2500.53 & 2625.93 \\
12800  & 13028.3 & 13260.6 & 13497.1 & 13737.7 & 13982.5 & 14231.6 & 14484.9 & 14742.7 & 15004.9 & 15271.6 \\
102400 & 103042  & 103688  & 104338  & 104992  & 105650  & 106313  & 106979  & 107650  & 108324  & 109003  \\
\hline
\end{tabular*}
\end{table}

\begin{figure}[!ht]
  \centering
  \includegraphics[width=.95\textwidth,height=.35\textwidth]{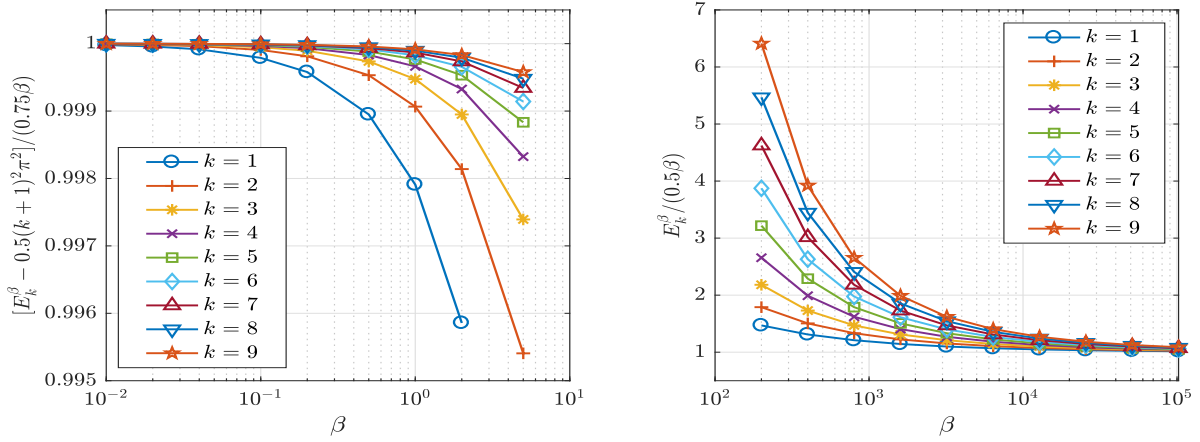}
  \caption{Asymptotic behaviors of the energies for index-$k$ ($k=1,2,\cdots,9$) excited states for the weakly (left) and strongly (right) repulsive interaction regime in Example~\ref{ex:es1d:box:vb}.}
  \label{fig:es1d_box_idxxx_es_asymptotics}
\end{figure}

From the experimental results that are partially shown in Figs.~\ref{fig:es1d_box_bxxx_idxxx}-\ref{fig:es1d_box_idxxx_es_asymptotics} and Table~\ref{tab:box:es-mus}, we have the following numerical observations:

\begin{enumerate}[(i)]
  \item Fig.~\ref{fig:es1d_box_bxxx_idxxx} shows that the index-$k$ excited state $\phi_k^\beta$ is oddly symmetric for $k=1,3,5,7,9$ (i.e., $k$ is odd) and evenly symmetric for $k=2,4,6,8$ (i.e., $k$ is even) with respect to the line $x=1/2$. For relatively small $\beta$, the profile of $\phi_k^\beta(x)$ is similar to that of $\phi_k^0(x)=\phi^{\mathrm{box}}_k(x)=\sqrt{2}\sin((k+1)\pi x)$. When $\beta$ is large, $\phi_k^\beta(x)$ has precisely two boundary layers and $k$ interior layers distributed equidistantly. It looks like a piecewise two-valued function that evenly takes $+1$ and $-1$.
  \item Table~\ref{tab:box:es-mus} shows that, for any $\beta\geq0$, the excited state with higher Morse index possesses higher energy and larger chemical potential, namely,
\[
  E(\phi_g^\beta)<E(\phi_1^\beta)<E(\phi_2^\beta)<\cdots \quad\Longleftrightarrow\quad
  \mu(\phi_g^\beta)<\mu(\phi_1^\beta)<\mu(\phi_2^\beta)<\cdots.
\]
Furthermore, for fixed $k=1,2,\cdots$, we observe that
\[
  \lim_{\beta\to+\infty}\frac{E(\phi_k^\beta)}{E(\phi_g^\beta)}=1,\quad
  \lim_{\beta\to+\infty}\frac{\mu(\phi_k^\beta)}{\mu(\phi_g^\beta)}=1 \quad \mbox{and} \quad
  \lim_{\beta\to+\infty}\frac{\mu(\phi_k^\beta)}{E(\phi_k^\beta)}=2.
\]
  \item From Fig.~\ref{fig:es1d_box_idxxx_es_asymptotics}, one observes that for the weakly repulsive interaction regime,
\[
  E(\phi_k^\beta)=\frac{(k+1)^2\pi^2}{2}+\frac34\beta +o(\beta) = E(\phi_k^{\mathrm{box}}) +o(\beta),
\]
where $\phi_k^{\mathrm{box}}(x)=\sqrt{2}\sin((k+1)x)$, while for the strongly repulsive interaction regime, $E(\phi_k^\beta)\approx\beta/2$.
\end{enumerate}

These observations are consistent with the results in \cite{BL2009AMS,BLZ2007BIMAS}.

\begin{example}\label{ex:es1d:ho:vb}
We now compute the first few excited states for the harmonic oscillator potential $V(x)=V_{\mathrm{ho}}(x)$ \eqref{eq:Vho} in 1D with various interaction coefficient $\beta$ and study the asymptotics of their energies and chemical potentials. 
\end{example}

The computational domain is taken as $U=[-16,16]$. Let $\phi_k^\beta$ and $\phi_g^\beta$ be the numerical index-$k$ excited state and ground state, respectively, for specified $\beta$. The initial data for $\phi_k^\beta$ and $\phi_g^\beta$ are, respectively, chosen as $\phi^0(x)=\phi^{\mathrm{ho}}_k(x)$, $v_i^0(x)=\phi^{\mathrm{ho}}_{i-1}(x)$, $i=1,2,\ldots,k-1$, and $\phi^0(x)=\phi^{\mathrm{ho}}_0(x)$. Fig.~\ref{fig:es1d_ho_bxxx_idxxx} plots the profiles of $\phi_k^{\beta}(x)$, $k=1,2,\cdots,9$, with different $\beta=100,400,1600$. The energies and chemical potentials of $\phi_g^\beta$ and $\phi_k^\beta$ ($k=1,2,\cdots,9$) for various $\beta$ are listed in Table~\ref{tab:box:es-mus}. Fig.~\ref{fig:es1d_ho_idxxx_es_asymptotics} shows that the asymptotics of the energies of $\phi_k^\beta$ ($k=1,2,\cdots,9$) in both the weakly and strongly repulsive interaction regime.

\begin{figure}[!ht]
  \centering
  \includegraphics[width=.3\textwidth]{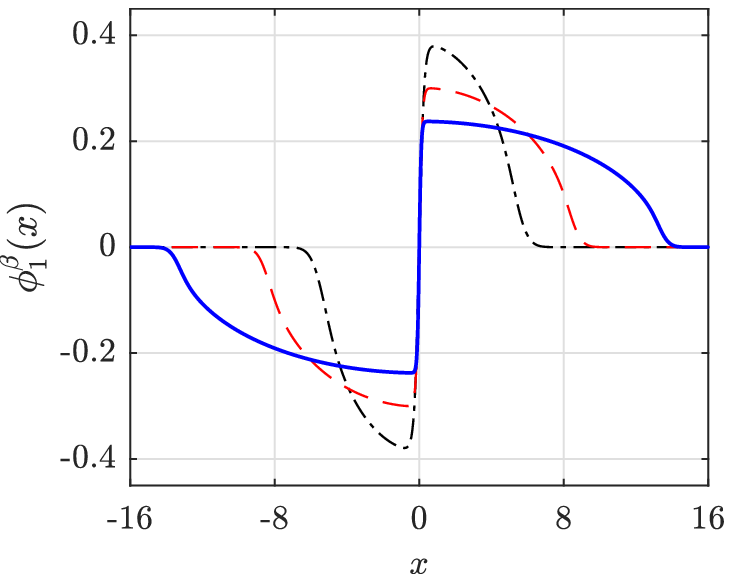}
  \includegraphics[width=.3\textwidth]{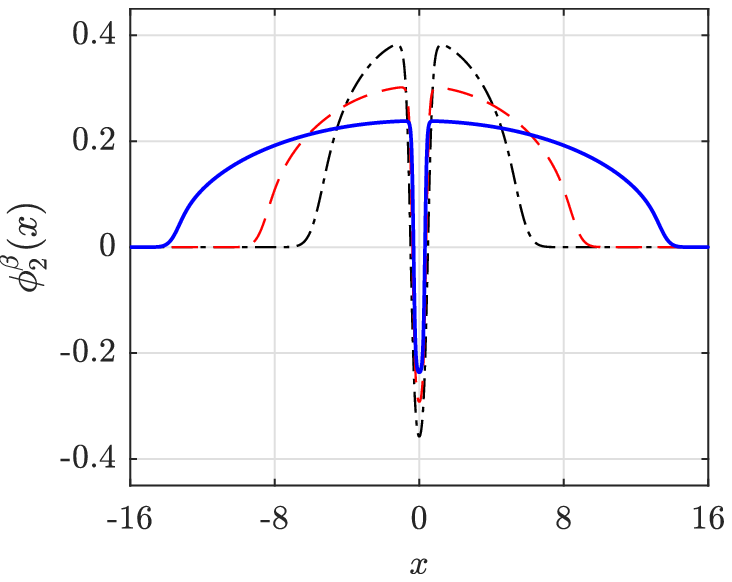}
  \includegraphics[width=.3\textwidth]{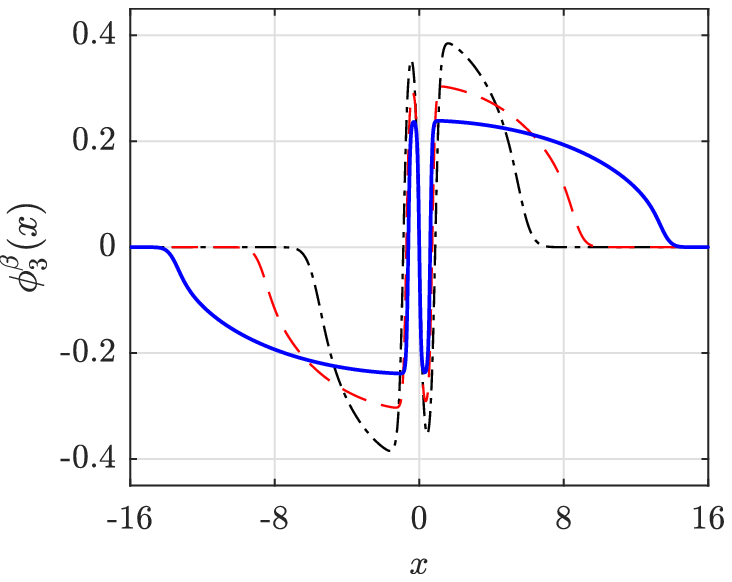} \\
  \includegraphics[width=.3\textwidth]{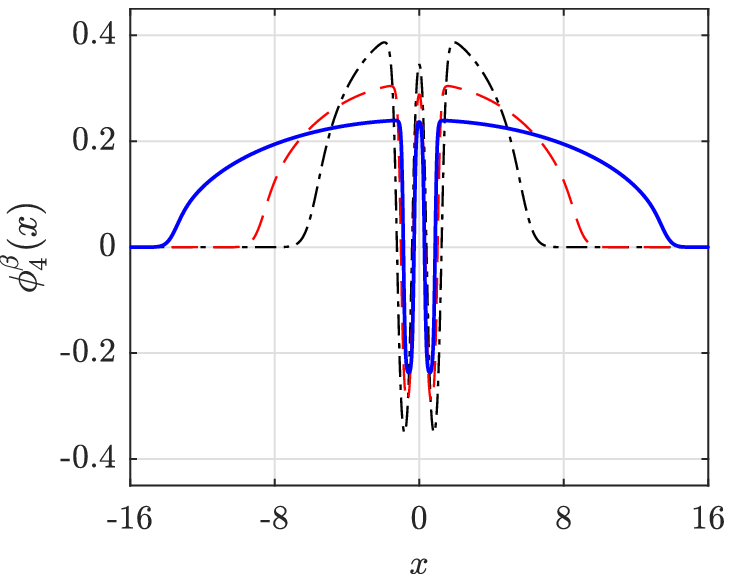}
  \includegraphics[width=.3\textwidth]{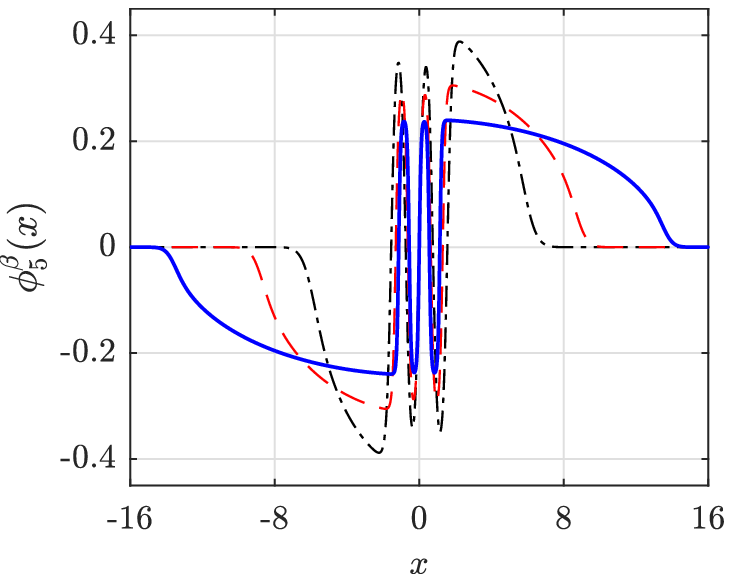}
  \includegraphics[width=.3\textwidth]{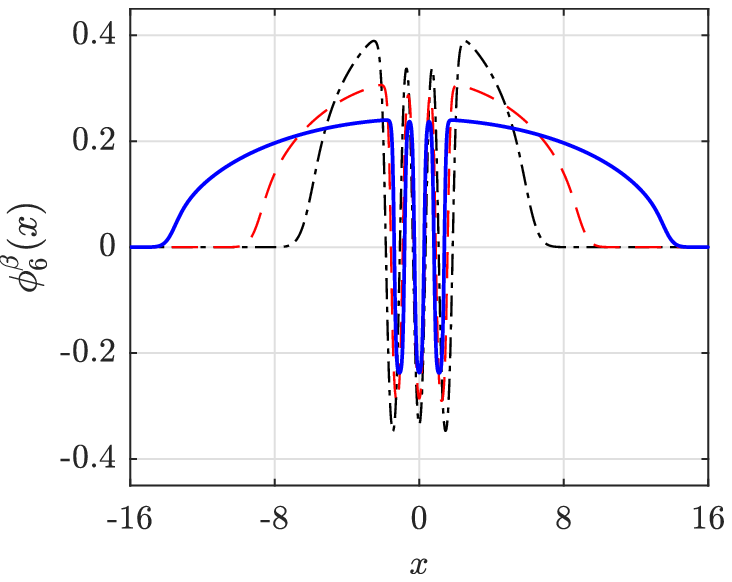} \\
  \includegraphics[width=.3\textwidth]{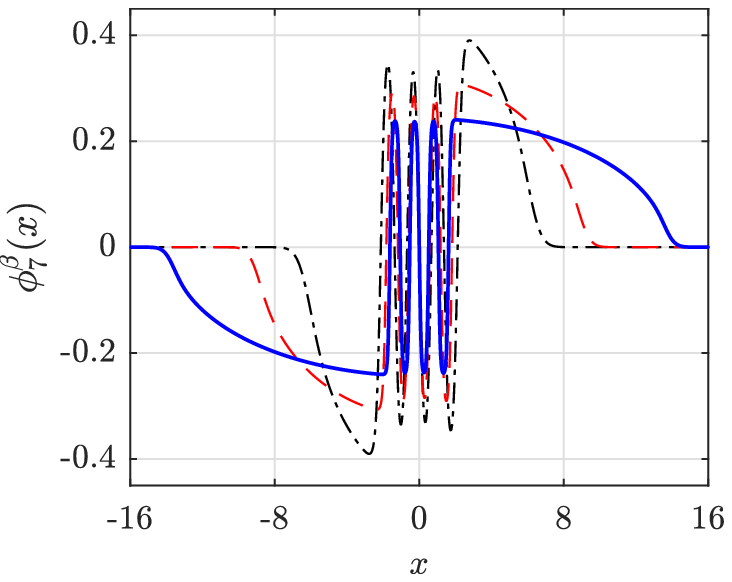}
  \includegraphics[width=.3\textwidth]{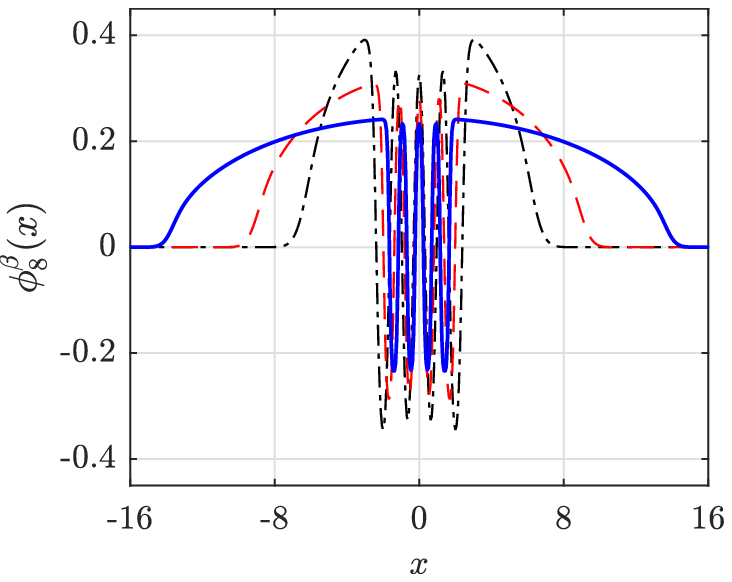}
  \includegraphics[width=.3\textwidth]{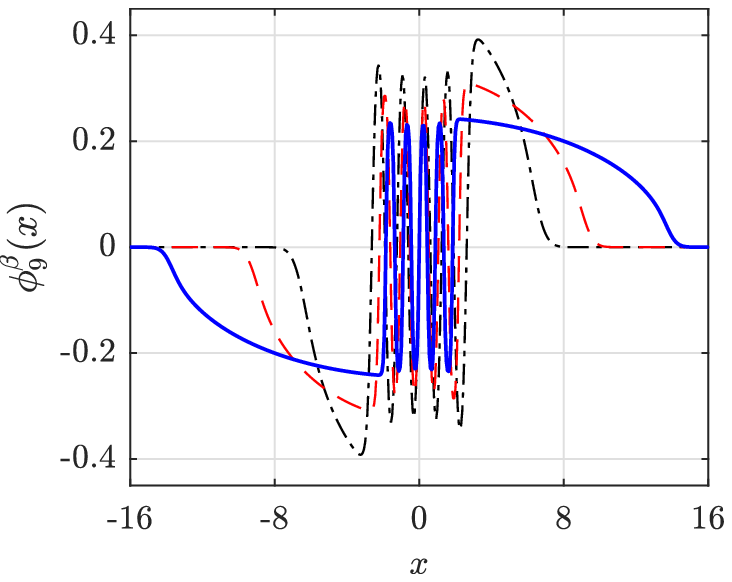}
  \caption{Profiles of index-$k$ ($k=1,2,\ldots,9$) excited states $\phi_k^\beta(x)$ for $\beta=100$ (black dash-dot lines), $400$ (red dash lines) and $1600$ (blue solid lines) in Example~\ref{ex:es1d:ho:vb}.}
  \label{fig:es1d_ho_bxxx_idxxx}
\end{figure}

\begin{table}[!ht]\footnotesize
  \centering
  \caption{Energies and chemical potentials of the ground state $\phi_g^\beta$ and excited states $\phi_k^\beta$ ($k=1,2,\cdots,9$) versus the interaction coefficient $\beta$ in Example~\ref{ex:es1d:ho:vb}.}
  \label{tab:ho:es-mus}
\begin{tabular*}{\textwidth}{@{\extracolsep{\fill}}cllllllllll}
\hline
$\beta$ & $E(\phi_g^\beta)$ & $E(\phi_1^\beta)$ & $E(\phi_2^\beta)$ & $E(\phi_3^\beta)$ & $E(\phi_4^\beta)$ & $E(\phi_5^\beta)$ & $E(\phi_6^\beta)$ & $E(\phi_7^\beta)$ & $E(\phi_8^\beta)$ & $E(\phi_9^\beta)$ \\
\hline
0    & 0.50000 & 1.50000 & 2.50000 & 3.50000 & 4.50000 & 5.50000 & 6.50000 & 7.50000 & 8.50000 & 9.50000 \\
0.01 & 0.50199 & 1.50150 & 2.50128 & 3.50115 & 4.50105 & 5.50098 & 6.50093 & 7.50088 & 8.50084 & 9.50081 \\
1    & 0.68948 & 1.64655 & 2.62626 & 3.61361 & 4.60467 & 5.59787 & 6.59246 & 7.58800 & 8.58424 & 9.58101 \\
10   & 1.94713 & 2.76538 & 3.64568 & 4.55841 & 5.49090 & 6.43654 & 7.39147 & 8.35325 & 9.32029 & 10.2914 \\
100  & 8.50853 & 9.24191 & 10.0079 & 10.7989 & 11.6100 & 12.4378 & 13.2797 & 14.1338 & 14.9985 & 15.8725 \\
400  & 21.3601 & 22.0777 & 22.8116 & 23.5594 & 24.3196 & 25.0909 & 25.8721 & 26.6626 & 27.4614 & 28.2680 \\
1600 & 53.7855 & 54.4968 & 55.2154 & 55.9407 & 56.6723 & 57.4098 & 58.1528 & 58.9011 & 59.6545 & 60.4127 \\
\hline
$\beta$ & $\mu(\phi_g^\beta)$ & $\mu(\phi_1^\beta)$ & $\mu(\phi_2^\beta)$ & $\mu(\phi_3^\beta)$ & $\mu(\phi_4^\beta)$ & $\mu(\phi_5^\beta)$ & $\mu(\phi_6^\beta)$ & $\mu(\phi_7^\beta)$ & $\mu(\phi_8^\beta)$ & $\mu(\phi_9^\beta)$ \\
\hline
0    & 0.50000 & 1.50000 & 2.50000 & 3.50000 & 4.50000 & 5.50000 & 6.50000 & 7.50000 & 8.50000 & 9.50000 \\
0.01 & 0.50398 & 1.50299 & 2.50256 & 3.50229 & 4.50211 & 5.50197 & 6.50186 & 7.50177 & 8.50169 & 9.50162 \\
1    & 0.86994 & 1.79015 & 2.75102 & 3.72629 & 4.70870 & 5.69528 & 6.68456 & 7.67572 & 8.66825 & 9.66182 \\
10   & 3.10724 & 3.86320 & 4.68057 & 5.53782 & 6.42244 & 7.32672 & 8.24566 & 9.17583 & 10.1148 & 11.0610 \\
100  & 14.1343 & 14.8505 & 15.5846 & 16.3352 & 17.1008 & 17.8799 & 18.6713 & 19.4739 & 20.2868 & 21.1092 \\
400  & 35.5775 & 36.2881 & 37.0061 & 37.7313 & 38.4636 & 39.2026 & 39.9480 & 40.6998 & 41.4576 & 42.2212 \\
1600 & 89.6319 & 90.3404 & 91.0518 & 91.7662 & 92.4834 & 93.2035 & 93.9265 & 94.6523 & 95.3809 & 96.1123 \\
\hline
\end{tabular*}
\end{table}

\begin{figure}[!ht]
  \centering
  \includegraphics[width=.95\textwidth,height=.35\textwidth]{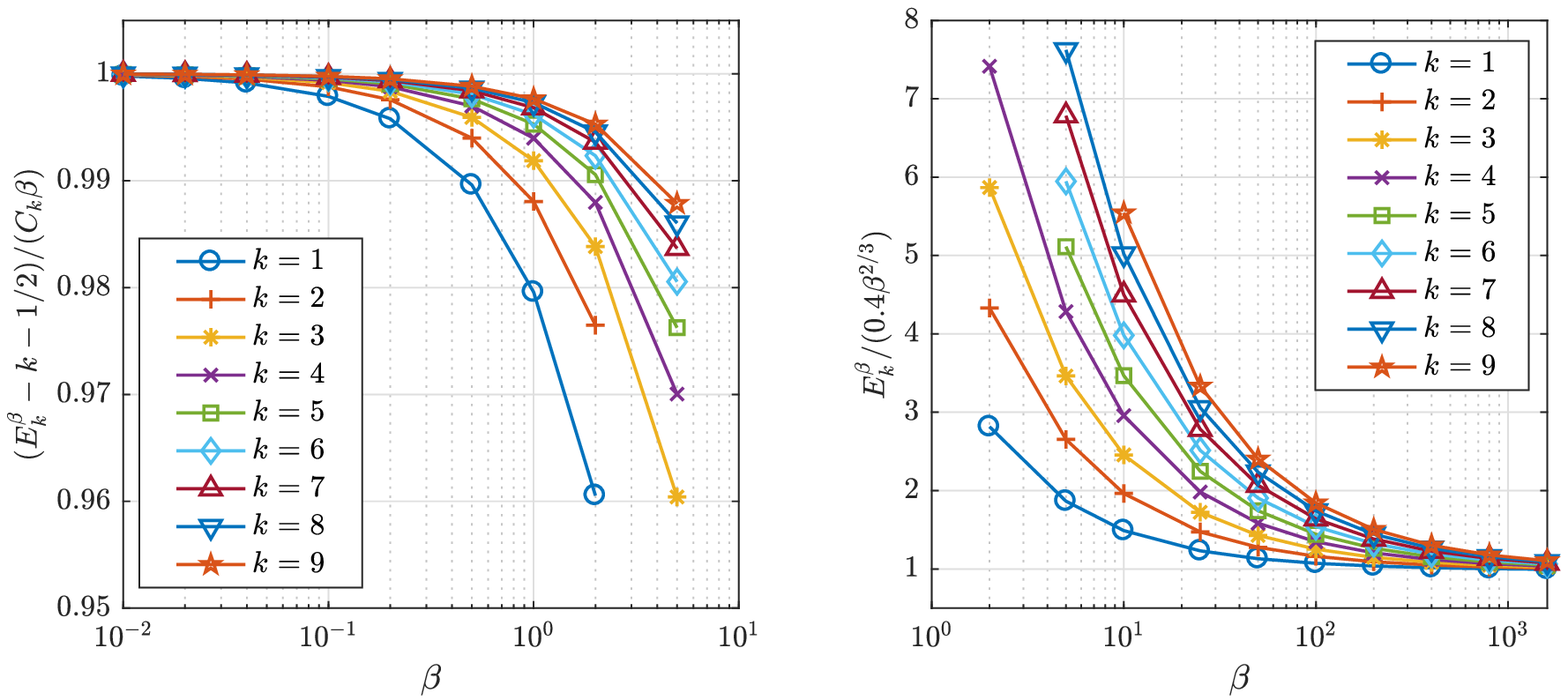}
  \caption{Asymptotic behaviors of the energy for index-$k$ ($k=1,2,\cdots,9$) excited states for the weakly (left) and strongly (right) repulsive interation regime in Example~\ref{ex:es1d:ho:vb}.}
  \label{fig:es1d_ho_idxxx_es_asymptotics}
\end{figure}

From the experimental results that are partially shown in Figs.~\ref{fig:es1d_ho_bxxx_idxxx}-\ref{fig:es1d_ho_idxxx_es_asymptotics} and Table~\ref{tab:ho:es-mus}, we have the following numerical observations:
\begin{enumerate}[(i)]
  \item  Fig.~\ref{fig:es1d_ho_bxxx_idxxx} shows that the index-$k$ excited state $\phi_k^\beta$ is precisely an odd function when $k$ is odd and an even function when $k$ is even. For relatively small $\beta$, the profile of $\phi_k^\beta(x)$ is similar to that of $\phi_k^0(x)=\phi^{\mathrm{ho}}_k(x)$. When $\beta$ is large, $\phi_k^\beta(x)$ has exactly $k$ interior layers or oscillations distributed densely near the center of domain, i.e., $x=0$, thus the multiscale structures are observed.
  \item Table~\ref{tab:ho:es-mus} shows that, for any $\beta\geq0$, all of the excited states we obtain have higher energies than that of the ground state. Moreover, the higher Morse indices the excited states have, the higher energy levels they possess. This observation is also available for the relationship between the Morse indices and the chemical potentials of excited states. That is
\[
  E(\phi_g^\beta)<E(\phi_1^\beta)<E(\phi_2^\beta)<\cdots \quad\Longleftrightarrow\quad
  \mu(\phi_g^\beta)<\mu(\phi_1^\beta)<\mu(\phi_2^\beta)<\cdots.
\]
Meanwhile, for fixed $k=1,2,\cdots$, we observe that
\[
  \lim_{\beta\to\infty}\frac{E(\phi_k^\beta)}{E(\phi_g^\beta)}=1,\quad
  \lim_{\beta\to\infty}\frac{\mu(\phi_k^\beta)}{\mu(\phi_g^\beta)}=1,\quad
  \lim_{\beta\to+\infty}\frac{\mu(\phi_k^\beta)}{E(\phi_k^\beta)}=\frac{5}{3}.
\]
  \item Fig.~\ref{fig:es1d_ho_idxxx_es_asymptotics} shows that, for the weakly repulsive interaction regime,
\[
  E(\phi_k^\beta)=k+\frac12 + C_k\beta + o(\beta) = E(\phi^{\mathrm{ho}}_k) + o(\beta),
\]
with $C_k=\frac{1}{2}\int_{\mathbb{R}}|\phi^{\mathrm{ho}}_k(x)|^4\mathrm{d}x$, whereas for the strongly interaction regime, $E(\phi_k^\beta)\approx\frac{2}{5}\beta^{2/3}$.
\end{enumerate}

These observations are consistent with the results in \cite{BL2009AMS,BLZ2007BIMAS}.

\subsubsection{Numerical results in 2D}

\begin{example}\label{ex:es2d}
In this example, we compute excited states in 2D BECs for the following four cases with various $\beta$. \\
\makebox[9ex][l]{\rm Case I.} $V(x,y)=V_{\mathrm{box}}(x,y)$ \eqref{eq:Vbox}, $U=[0,1]^2$; \\
\makebox[9ex][l]{\rm Case II.} $V(x,y)=V_{\mathrm{ho}}(x,y)$ \eqref{eq:Vho}, $U=[-10,10]^2$;  \\
\makebox[9ex][l]{\rm Case III.} $V(x,y)=V_{\mathrm{hol}}(x,y)$ \eqref{eq:Vhol} with $\kappa=25$, $U=[-10,10]^2$; \\
\makebox[9ex][l]{\rm Case IV.}  $V(x,y)=V_{\mathrm{hol}}(x,y)$ \eqref{eq:Vhol} with $\kappa=50$, $U=[-10,10]^2$.
\end{example}

As suggested by subsection~\ref{sec:esbec-b0}, the information of initial guesses is given in Table~\ref{tab:es2d:igs}. We compute the ground state (by the normalized gradient flow \cite{BD2004SISC,LC2021SISC}) and a few excited states for four cases with various $\beta=0,10,50,100,500,1000$. Tables~\ref{tab:es2d:emu-i}-\ref{tab:es2d:emu-iv} list the energies and chemical potentials of these solutions. Fig.~\ref{fig:es2d_boxhool_b1000_4x5} plots the pseudo-color images of excited states with $\beta=1000$.

\begin{table}[!ht]\footnotesize
  \centering
  \caption{Initial guesses in Example~\ref{ex:es2d}. $\varphi_\mathbf{j}=\phi_\mathbf{j}^{\mathrm{box}}$ \eqref{eq:b0eigpairs-box} for Case I and $\varphi_\mathbf{j}=\phi_\mathbf{j}^{\mathrm{ho}}$ \eqref{eq:b0eigpairs-ho} for Cases II-IV.}
  \label{tab:es2d:igs}
\begin{tabular*}{\textwidth}{@{\extracolsep{\fill}}ccll}
    \hline
    solution & $k$ (index) & initial guess for $\phi$ & initial guess for $(v_1,\ldots,v_k)$ \\ \hline
    $\phi_g$       & 0 & $\varphi_{(0,0)}$ & -- \\ 
    $\phi_{10}$    & 1 & $\varphi_{(1,0)}$ & $\varphi_{(0,0)}$ \\ 
    $\phi_{01}$    & 1 & $\varphi_{(0,1)}$ & $\varphi_{(0,0)}$ \\ 
    $\phi_{10+01}$ & 1 & $[\varphi_{(1,0)}+\varphi_{(0,1)}]/\sqrt{2}$ & $\varphi_{(0,0)}$ \\ 
    $\phi_{10-01}$ & 1 & $[\varphi_{(1,0)}-\varphi_{(0,1)}]/\sqrt{2}$ & $\varphi_{(0,0)}$\\ 
    $\phi_{11}$    & 3 & $\varphi_{(1,1)}$ & $(\varphi_{(0,0)},\varphi_{(1,0)},\varphi_{(0,1)})$ \\
    \hline
\end{tabular*}
\end{table}

\begin{table}[!ht]\footnotesize
  \centering
  \caption{Energies and chemical potentials of ground and excited states for Case I in Example~\ref{ex:es2d}.}
  \label{tab:es2d:emu-i}
  \begin{tabular*}{\textwidth}{@{\extracolsep{\fill}}ccccccc}
\hline
$\beta$
 & $E_g$       
 & $E_{10}$    
 & $E_{01}$    
 & $E_{10+01}$ 
 & $E_{10-01}$ 
 & $E_{11}$  \\  \hline
 0    & 9.8696  & 24.6740 & 24.6740 & 24.6740 & 24.6740 & 39.4784 \\ 
 10   & 19.4655 & 34.7611 & 34.7611 & 36.3205 & 36.3205 & 50.1222 \\ 
 50   & 49.2110 & 67.5593 & 67.5593 & 72.1768 & 72.1768 & 86.3251 \\ 
 100  & 81.8684 & 103.473 & 103.473 & 110.034 & 110.034 & 125.648 \\ 
 500  & 314.632 & 351.897 & 351.897 & 365.667 & 365.667 & 389.910 \\ 
 1000 & 589.286 & 638.718 & 638.718 & 657.680 & 657.680 & 688.933 \\ 
 \hline
$\beta$
& $\mu_g$
& $\mu_{10}$
& $\mu_{01}$
& $\mu_{10+01}$
& $\mu_{10-01}$
& $\mu_{11}$ \\ \hline
0    & 9.8696  & 24.6740 & 24.6740 & 24.6740 & 24.6740 & 39.4784 \\ 
10   & 28.0732 & 44.0760 & 44.0760 & 46.9070 & 46.9070 & 60.2603 \\ 
50   & 83.3738 & 105.336 & 105.336 & 112.458 & 112.458 & 127.971 \\ 
100  & 145.019 & 172.513 & 172.513 & 182.220 & 182.220 & 200.761 \\ 
500  & 594.368 & 646.225 & 646.225 & 666.306 & 666.306 & 698.892 \\ 
1000 & 1131.39 & 1201.69 & 1201.69 & 1229.47 & 1229.47 & 1272.81 \\ 
\hline
  \end{tabular*}
\end{table}

\begin{table}[!ht]\footnotesize
  \centering
  \caption{Energies and chemical potentials of ground and excited states for Case II in Example~\ref{ex:es2d}.}
  \label{tab:es2d:emu-ii}
   \begin{tabular*}{\textwidth}{@{\extracolsep{\fill}}ccccccc}
\hline
$\beta$
 & $E_g$       
 & $E_{10}$    
 & $E_{01}$    
 & $E_{10+01}$ 
 & $E_{10-01}$ 
 & $E_{11}$  \\  \hline
 0   & 1.0000  & 2.0000  & 2.0000  & 2.0000  & 2.0000  & 3.0000  \\
 10  & 1.5923  & 2.4916  & 2.4916  & 2.4916  & 2.4916  & 3.4003  \\
 50  & 2.8960  & 3.7111  & 3.7111  & 3.7111  & 3.7111  & 4.5283  \\
 100   & 3.9459  & 4.7329  & 4.7329  & 4.7329  & 4.7329  & 5.5204  \\
 500   & 8.5118  & 9.2567  & 9.2567  & 9.2567  & 9.2567  & 10.0014  \\
 1000  & 11.9718  & 12.7059  & 12.7059  & 12.7059  & 12.7059  & 13.4399  \\
 \hline
$\beta$
& $\mu_g$
& $\mu_{10}$
& $\mu_{01}$
& $\mu_{10+01}$
& $\mu_{10-01}$
& $\mu_{11}$ \\ \hline
 0   & 1.0000  & 2.0000  & 2.0000  & 2.0000  & 2.0000  & 3.0000  \\
 10  & 2.0638  & 2.9094  & 2.9094  & 2.9094  & 2.9094  & 3.7618  \\
 50  & 4.1430  & 4.9128  & 4.9128  & 4.9128  & 4.9128  & 5.6813  \\
 100   & 5.7598  & 6.5109  & 6.5109  & 6.5109  & 6.5109  & 7.2613  \\
 500   & 12.6783  & 13.4051  & 13.4051  & 13.4051  & 13.4051  & 14.1317  \\
 1000  & 17.8886  & 18.6097  & 18.6097  & 18.6097  & 18.6097  & 19.3306  \\
\hline
  \end{tabular*}
\end{table}

\begin{table}[!ht]\footnotesize
  \centering
  \caption{Energies and chemical potentials of ground and excited states for Case III in Example~\ref{ex:es2d}.}
  \label{tab:es2d:emu-iii}
\begin{tabular*}{\textwidth}{@{\extracolsep{\fill}}ccccccc}
\hline
$\beta$
 & $E_g$       
 & $E_{10}$    
 & $E_{01}$    
 & $E_{10+01}$ 
 & $E_{10-01}$ 
 & $E_{11}$  \\  \hline
 0   & 5.4894  & 10.8158  & 10.8158  & 10.8158  & 10.8158  & 16.1421  \\
 10  & 8.6291  & 13.0150  & 13.0150  & 12.8903  & 12.8903  & 17.9353  \\
 50  & 13.4615  & 16.4307  & 16.4307  & 15.4508  & 15.4508  & 20.1165  \\
 100   & 16.0172  & 18.7438  & 18.7438  & 17.5627  & 17.5627  & 21.6350  \\
 500   & 24.5175  & 26.3802  & 26.3802  & 25.6729  & 25.6729  & 28.2840  \\
 1000  & 29.8150  & 31.4142  & 31.4142  & 30.8570  & 30.8570  & 33.0400  \\
 \hline
$\beta$
& $\mu_g$
& $\mu_{10}$
& $\mu_{01}$
& $\mu_{10+01}$
& $\mu_{10-01}$
& $\mu_{11}$ \\ \hline
 0   & 5.4894  & 10.8158  & 10.8158  & 10.8158  & 10.8158  & 16.1421  \\
 10  & 11.0942  & 14.3353  & 14.3353  & 13.9468  & 13.9468  & 18.9615  \\
 50  & 16.7198  & 19.4602  & 19.4602  & 17.9147  & 17.9147  & 21.9572  \\
 100   & 20.2267  & 22.3176  & 22.3176  & 21.2398  & 21.2398  & 24.2538  \\
 500   & 31.2499  & 32.7361  & 32.7361  & 32.2639  & 32.2639  & 34.1864  \\
 1000  & 38.3834  & 39.6557  & 39.6557  & 39.2391  & 39.2391  & 41.0250  \\
\hline
\end{tabular*}
\end{table}

\begin{table}[!ht]\footnotesize
  \centering
  \caption{Energies and chemical potentials of ground and excited states for Case IV in Example~\ref{ex:es2d}.}
  \label{tab:es2d:emu-iv}
\begin{tabular*}{\textwidth}{@{\extracolsep{\fill}}ccccccc}
\hline
$\beta$
 & $E_g$       
 & $E_{10}$    
 & $E_{01}$    
 & $E_{10+01}$ 
 & $E_{10-01}$ 
 & $E_{11}$  \\  \hline
 0  & 7.7626  & 15.3637  & 15.3637  & 15.3637  & 15.3637  & 22.9649  \\
 10  & 12.2495  & 17.1385  & 17.1385  & 16.6037  & 16.6037  & 23.7793  \\
 50  & 17.6091  & 21.4544  & 21.4544  & 19.8357  & 19.8357  & 25.6103  \\
 100   & 20.8412  & 24.1152  & 24.1152  & 22.6530  & 22.6530  & 27.4948  \\
 500   & 32.2079  & 34.6044  & 34.6044  & 33.8998  & 33.8998  & 37.0849  \\
 1000  & 39.6188  & 41.7623  & 41.7623  & 41.0769  & 41.0769  & 43.9528  \\
 \hline
$\beta$
& $\mu_g$
& $\mu_{10}$
& $\mu_{01}$
& $\mu_{10+01}$
& $\mu_{10-01}$
& $\mu_{11}$ \\ \hline
 0   & 7.7626  & 15.3637  & 15.3637  & 15.3637  & 15.3637  & 22.9649  \\
 10  & 15.7812  & 18.6102  & 18.6102  & 17.5856  & 17.5856  & 24.2988  \\
 50  & 21.7101  & 24.9488  & 24.9488  & 23.2633  & 23.2633  & 27.6786  \\
 100   & 25.9323  & 28.4524  & 28.4524  & 27.4443  & 27.4443  & 30.9539  \\
 500   & 41.7854  & 43.8228  & 43.8228  & 43.2263  & 43.2263  & 46.1402  \\
 1000  & 51.2663  & 52.9797  & 52.9797  & 52.3203  & 52.3203  & 54.7172  \\
\hline
\end{tabular*}
\end{table}

\begin{figure}[!ht]
  \centering
  \includegraphics[width=.98\textwidth]{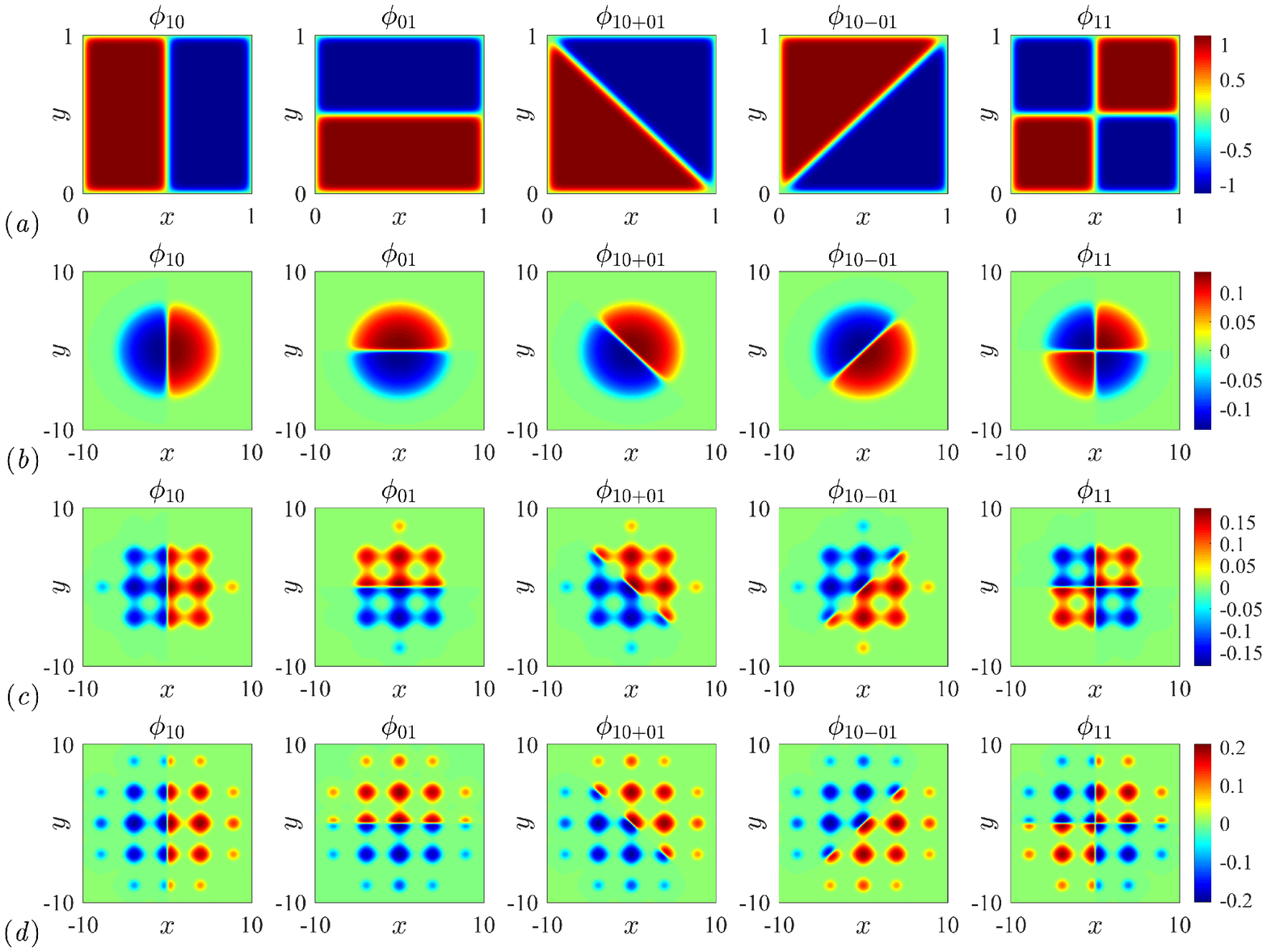}
  \caption{Four index-1 excited states $\phi_{10}(x,y)$ (left column), $\phi_{01}(x,y)$ (second column), $\phi_{10+01}(x,y)$ (third column), $\phi_{10-01}(x,y)$ (fourth column) and an index-3 excited state $\phi_{11}(x,y)$ (right column) with $\beta=1000$ in Example~\ref{ex:es2d}. $(a)\sim(d)$ for Cases I$\sim$IV, respectively.}
  \label{fig:es2d_boxhool_b1000_4x5}
\end{figure}

From the numerical results shown in Fig.~\ref{fig:es2d_boxhool_b1000_4x5}, Tables~\ref{tab:es2d:emu-i}-\ref{tab:es2d:emu-iv}, and additional experimental results not shown here, we have the following numerical observations:
\begin{enumerate}[(i)]
  \item From Tables~\ref{tab:es2d:emu-i}-\ref{tab:es2d:emu-iv}, we observe that for each case, $E(\phi_{10})=E(\phi_{01})$, $\mu(\phi_{10})=\mu(\phi_{01})$, $E(\phi_{10+01})=E(\phi_{10-01})$, and $\mu(\phi_{10+01})=\mu(\phi_{10-01})$. From Fig.~\ref{fig:es2d_boxhool_b1000_4x5}, the profiles of $\phi_{01}$ and $\phi_{10-01}$ can be obtained from that of $\phi_{10}$ and $\phi_{10+01}$, respectively, by a $90^\circ$ rotation. Moreover, some boundary/interior layers and multiscale structures are observed. It is worthwhile to point out that, the shape and symmetry of excited states are independent of the shape of domain if $V(\mathbf{x})$ is a harmonic or optical lattice potential and the computational domain is large enough so that the error of domain truncation can be ignored, whereas they are affected by the shape of domain if $V(\mathbf{x})$ is a box potential.
  \item Tables~\ref{tab:es2d:emu-i}-\ref{tab:es2d:emu-iv} show the following facts:
  \begin{itemize}
  	\item For Case I with $\beta>0$,
\begin{align*}
  & E(\phi_g)<E(\phi_{10})=E(\phi_{01})<E(\phi_{10+01})=E(\phi_{10-01})<E(\phi_{11}), \\
  & \mu(\phi_g)<\mu(\phi_{10})=\mu(\phi_{01})<\mu(\phi_{10+01})=\mu(\phi_{10-01})<\mu(\phi_{11}).
\end{align*}
\item For Case II with $\beta>0$ (or each case with $\beta=0$),
\begin{align*}
  & E(\phi_g)<E(\phi_{10})=E(\phi_{01})=E(\phi_{10+01})=E(\phi_{10-01})<E(\phi_{11}), \\
  & \mu(\phi_g)<\mu(\phi_{10})=\mu(\phi_{01})=\mu(\phi_{10+01})=\mu(\phi_{10-01})<\mu(\phi_{11}).
\end{align*}
\item For Cases III and IV with $\beta>0$,
\begin{align*}
  & E(\phi_g)<E(\phi_{10+01})=E(\phi_{10-01})<E(\phi_{10})=E(\phi_{01})<E(\phi_{11}), \\
  & \mu(\phi_g)<\mu(\phi_{10+01})=\mu(\phi_{10-01})<\mu(\phi_{10})=\mu(\phi_{01})<\mu(\phi_{11}).
\end{align*}
\end{itemize}
Consequently, for all cases in this example, the order of energies is consistent with that of chemical potentials of solutions we obtained. The first excited states are exactly index-1 excited states, but different index-1 excited states may have different energies and chemical potentials. Moreover, the index-3 excited state $\phi_{11}$ possess higher energy and larger chemical potential than those of index-1 excited states.
\end{enumerate}

These numerical results indicate that the Morse index of the excited state has a certain monotonous dependence on energy and chemical potential (i.e., the higher the index, the larger the energy and chemical potential), but generally there is no strict one-to-one correspondence.

\section{Concluding remarks}\label{sec:conclusion}
In this paper, a constrained gentlest ascent dynamics (CGAD) for finding general constrained saddle points with any specified Morse index was proposed. The linearly stable steady state of the CGAD was proved to be exactly a nondegenerate constrained saddle point with the corresponding index. The locally exponential convergence of an idealized CGAD around a nondegenerate constrained saddle point with the corresponding Morse index was also provided. Moreover, the CGAD was applied to compute some excited states of single-component Bose--Einstein condensates by finding constrained saddle points of the corresponding Gross--Pitaevskii energy functional under the normalization constraint. The properties of excited states were studied both mathematically and numerically. Extensive numerical results were reported to show the effectiveness and robustness of our method and demonstrate some interesting physics. It is worthwhile to point out that the CGAD can be applied to solve many other scientific problems. And, many optimization and preconditioning techniques can be used to further improve the computational efficiency of the CGAD. Some related works are ongoing.

\section*{Acknowledgments}
\addcontentsline{toc}{section}{Acknowledgments}
This work was supported by the NSFC grants 12101252, 12171148, 11971007, 11771138 and the innovation platform open fund of the Education Department in Hunan Province (18K025). The work of W. Liu was also partially supported by the International Postdoctoral Exchange Fellowship Program No. PC2021024 and the Guangdong Basic and Applied Basic Research Foundation grant 2022A1515010351.

\appendix

\section{Proof of \texorpdfstring{Lemma~\ref{lem:cgad-idxk-conpty}}{Lemma 3.2}}\label{sec:pf-cgad-idxk-conpty}
\normalcolor
For $l=1,2,\ldots,m$ and $i=1,2,\ldots,k$, applying \eqref{eq:cgad-idxk} and noting that $\langle G_l'(u),\hat{H}(u)v_i\rangle=0$, we have
\begin{align*}
  \frac{\mathrm{d}}{\mathrm{d}t}\langle G_l'(u),v_i\rangle 
 &= \langle G_l''(u)\dot{u},v_i\rangle + \langle G_l'(u),\dot{v}_i\rangle \\
 &= -\frac{1}{\gamma_0}\bigg\langle G_{l}''(u)v_i,\, F(u) - 2\sum_{j=1}^k\langle F(u),v_j\rangle v_j\bigg\rangle \\
 &\quad \; + \frac{1}{\gamma_i}\sum_{j=1}^i \lambda_{ij}\langle G_l'(u),v_j\rangle + \frac{1}{\gamma_i}\sum_{l'=1}^m \bar{\lambda}_{il'}\langle G_l'(u),G_{l'}'(u)\rangle.
\end{align*}
By the definition of $\bar{\lambda}_{il'}$ in \eqref{eq:cgad-idxk-lambdaibar}, which is equivalent to
$$
-\frac{1}{\gamma_0}\bigg\langle G_{l}''(u)v_i,\, F(u) - 2\sum_{j=1}^k\langle F(u),v_j\rangle v_j\bigg\rangle+\frac{1}{\gamma_i}\sum_{l'=1}^m \bar{\lambda}_{il'}\langle G_l'(u),G_{l'}'(u)\rangle=0,
$$
it holds,
\begin{align*}
  \frac{\mathrm{d}}{\mathrm{d}t}
  \begin{pmatrix}
  \langle G_l'(u),v_1\rangle \\
  \langle G_l'(u),v_2\rangle \\
  \vdots \\
  \langle G_l'(u),v_k\rangle \\
  \end{pmatrix} = 
  \begin{pmatrix}
  \tilde{\lambda}_{11} & 0 & \cdots & 0 \\
  \tilde{\lambda}_{21} & \tilde{\lambda}_{22} & \cdots & 0 \\
  \vdots & \vdots & \ddots & \vdots \\
  \tilde{\lambda}_{k1} & \tilde{\lambda}_{k2} & \cdots & \tilde{\lambda}_{kk} \\
  \end{pmatrix}
  \begin{pmatrix}
  \langle G_l'(u),v_1\rangle \\
  \langle G_l'(u),v_2\rangle \\
  \vdots \\
  \langle G_l'(u),v_k\rangle \\
  \end{pmatrix},\quad l=1,2,\ldots,m,
\end{align*}
with $\tilde{\lambda}_{ij}:=\lambda_{ij}/\gamma_i$ ($1\leq j\leq i\leq k$). Then the conclusion \eqref{eq:cgad-idxk-conpty2} follows from the initial condition \eqref{eq:cgad-idxk-init2}. Moreover, by using \eqref{eq:cgad-idxk} and \eqref{eq:cgad-idxk-conpty2}, and noting that $\langle G_l'(u),F(u)\rangle=0$, $l=1,2,\ldots,m$, we have
\begin{align*}
  \frac{\mathrm{d}}{\mathrm{d}t}G_l(u)
  = \langle G_l'(u),\dot{u}\rangle
  = \frac{2}{\gamma_0}\sum_{i=1}^k\langle F(u),v_i\rangle \langle G_l'(u),v_i\rangle
  =0,\quad l=1,2,\ldots,m.
\end{align*}
Thus, \eqref{eq:cgad-idxk-conpty1} is verified immediately from \eqref{eq:cgad-idxk-init1}. Furthermore, by using \eqref{eq:cgad-idxk}, \eqref{eq:cgad-idxk-lambdaij} and \eqref{eq:cgad-idxk-conpty2}, we have, for $1\leq j\leq i\leq k$,
\begin{align*}
  \frac{\mathrm{d}}{\mathrm{d} t}\left(\langle v_i,v_j\rangle -\delta_{ij}\right)
  &= \langle \dot{v}_i,v_j\rangle+\langle v_i,\dot{v}_j\rangle \\
  &= - \left(\frac{1}{\gamma_i}+\frac{1}{\gamma_j}\right)\langle \hat{H}(u)v_i,v_j\rangle
    + \frac{1}{\gamma_i}\sum_{l=1}^i \lambda_{il}\langle v_l,v_j\rangle
    + \frac{1}{\gamma_j}\sum_{l=1}^j \lambda_{jl}\langle v_l,v_i\rangle \\
  &= \begin{cases}
  \frac{2}{\gamma_i}\sum_{l=1}^i \lambda_{il}\left(\langle v_i,v_l\rangle-\delta_{il}\right), & j=i, \\
  \frac{1}{\gamma_i}\sum_{l=1}^i \lambda_{il}\left(\langle v_l,v_j\rangle-\delta_{jl}\right) +
  \frac{1}{\gamma_j}\sum_{l=1}^j \lambda_{jl}\left(\langle v_l,v_i\rangle-\delta_{il}\right), & j<i.
  \end{cases}
\end{align*}
Denote $\mathbf{y}$ as the vector of length $\frac{k(k+1)}{2}$ formed by $\left(\langle v_i,v_j\rangle -\delta_{ij}\right),1\leq j\leq i\leq k$. Then, we have $\mathbf{y}'(t)=\mathbf{A}(t)\mathbf{y}(t)$, where $\mathbf{A}$ is a matrix of degree $\frac{k(k+1)}{2}$, whose elements only depend on Lagrange multipliers $\lambda_{ij},1\leq j\leq i\leq k$ \eqref{eq:cgad-idxk-lambdaij} and relaxation constants. The initial condition $\mathbf{y}(0)=\mathbf{0}$ \eqref{eq:cgad-idxk-init} leads to $\mathbf{y}(t)\equiv\mathbf{0}$. That is \eqref{eq:cgad-idxk-conpty3}.
The proof is completed.

\section{Proof of \texorpdfstring{Lemma~\ref{lem:dmudF}}{Lemma 4.1}}\label{sec:pf-dmudF}
\normalcolor
Note that, by \eqref{eq:projgrad-def},
\begin{align*}
\sum_{j=1}^m \big\langle G_i'(u),G_j'(u)\big\rangle \mu_j(u)
= \big\langle G_i'(u),E'(u)\big\rangle,\quad i=1,2,\ldots,m.
\end{align*}
Differentiating in both sides of the above equation and applying the definitions of $F(u)$ \eqref{eq:projgrad-def} and $H(u)$, we obtain
\begin{align*}
\sum_{j=1}^m \big\langle G_i'(u),G_j'(u)\big\rangle \mu_j'(u)
&= G_i''(u)E'(u)+E''(u)G_i'(u)  -\sum_{j=1}^m \mu_j(u)\left[G_i''(u)G_j'(u)+G_j''(u)G_i'(u)\right] \\
&= G_i''(u)F(u)+H(u)G_i'(u), \quad i=1,2,\ldots,m.
\end{align*}
Thus
\begin{align*}
\mu_i'(u)=\sum_{j=1}^m g_{ij}(u)\left[G_j''(u)F(u)+H(u)G_j'(u)\right], \quad i=1,2,\ldots,m.
\end{align*}
For any $v\in T_u\mathcal{M}$, applying definitions of $P_u$ \eqref{eq:Pu-def} and $\hat{H}(u)$ \eqref{eq:Hhat-def}, yields
\begin{align*}
F'(u)v 
&= E''(u)v-\sum_{i=1}^m\mu_i(u)G_i''(u)v-\sum_{i=1}^m\big\langle \mu_i'(u),v\big\rangle G_i'(u) \\
&= H(u)v-\sum_{i=1}^m\sum_{j=1}^m g_{ij}(u) \left[\big\langle G_j''(u)F(u),v\big\rangle+\big\langle H(u)G_j'(u),v\big\rangle\right] G_i'(u) \\
&=\hat{H}(u)v-\sum_{i=1}^m\sum_{j=1}^m g_{ij}(u) \big\langle G_j''(u)F(u),v\big\rangle G_i'(u),
\end{align*}
where the self-adjointness of $H(u)$ and the fact $P_uv=v$ are used.
The proof is completed.

\footnotesize
\bibliographystyle{abbrv}
\bibliography{cgad_refs}

\end{document}